\documentclass[12pt]{article}
\usepackage{amsthm}
\usepackage{amsfonts}
\usepackage{epsfig}
\newcommand{\FF}{{\cal F}}
\newtheorem{theorem}{Theorem}
\newtheorem{corollary}[theorem]{Corollary}
\newtheorem{lemma}[theorem]{Lemma}

\newcommand{\claim}[2]{\begin{equation}\mbox{\parbox{\linewidth}{{\em #2}}}\label{#1}\end{equation}}

\title{Choosability of planar graphs of girth $5$}
\author{Zden\v{e}k Dvo\v{r}\'ak\thanks{Charles University, Prague, Czech Republic. E-mail: {\tt rakdver@kam.mff.cuni.cz}.
Supported by Institute for Theoretical Computer Science, project 1M0545 of Ministry of~Education of~Czech Republic.}\and
Ken-ichi Kawarabayashi\thanks{National Institute of Informatics, Tokyo, Japan. E-mail: {\tt k\_keniti@nii.ac.jp}.}}
\date{}

\begin{document}
\maketitle

\begin{abstract}
Thomassen proved that any plane graph of girth $5$ is list-colorable
from any list assignment such that all vertices have lists of size two or three
and the vertices with list of size two are all incident with the outer face
and form an independent set.  We present a strengthening of this result,
relaxing the constraint on the vertices with list of size two.
This result is used to bound the size of the 3-list-coloring critical plane graphs
with one precolored face.  
\end{abstract}

\section{Introduction}
All graphs considered in this paper are simple and finite. The concepts 
of list coloring and choosability were introduced by Vizing~\cite{vizing1976} and independently by 
Erd\H{o}s et al.~\cite{erdosrubintaylor1979}:
A {\em list assignment} of $G$ is a function $L$ that assigns to each vertex 
$v \in V(G)$ a list $L(v)$ of colors. An {\em $L$-coloring} is a function 
$\varphi: V(G) \rightarrow \bigcup_v L(v)$ such 
that $\varphi(v) \in L(v)$ for every $v \in V(G)$ and
$\varphi(u) \neq \varphi(v)$ whenever $u, v$ are adjacent vertices 
of $G$. If $G$ admits an $L$-coloring, then it is {\em $L$-colorable}. 
A graph $G$ is {\em $k$-choosable} if it is $L$-colorable for
for every list assignment $L$ such that $|L(v)| \ge k$ 
for all $v \in V(G)$.

A well-known result of Gr\"otzsch~\cite{grotzsch} states that any triangle-free planar graph is $3$-colorable.
Since the cycles of length $4$ can be easily eliminated, the main part of the proof of Gr\"otzsch's theorem
concerns graphs of girth $5$.  Generalizing this result, Thomassen~\cite{thomassen1995-34} proved
that every planar graph of girth at least $5$ is $3$-choosable.  In fact, he proved the following stronger claim:

\begin{theorem}\label{thm-thom}
Let $G$ be a plane graph of girth at least $5$ and $F$ a
face of $G$.  Let $P$ be a path in $G$ of length at most $5$,
such that $V(P)\subseteq V(F)$.  Let $L$ be an
assignment of lists to the vertices of $G$
such that
$|L(v)|=3$ for $v\in V(G)\setminus V(F)$,
$|L(v)|\ge 2$ for $v\in V(F)\setminus V(P)$,
$|L(v)|=1$ for $v\in V(P)$,
the lists of vertices of $P$ give a proper coloring of the
subgraph induced by $V(P)$, and
a vertex $v$ with $|L(v)|=2$ is not adjacent
to any vertex $u$ with $|L(u)|\le 2$.
Then, $G$ can be $L$-colored.
\end{theorem}

Voigt~\cite{voigt1995} found a triangle-free planar graph that is not $3$-choosable,
thus the restriction on the girth of the graph in Theorem~\ref{thm-thom} cannot be
relaxed without imposing further constraints on $4$-cycles.  Such a generalizations exist,
e.g., Dvo\v{r}\'ak et al.~\cite{fc367} proved that Theorem~\ref{thm-thom} holds for
triangle-free graphs as long as no $4$-cycle shares an edge with a cycle of length at most $5$.

In this paper, we study another approach to strenghtening Theorem~\ref{thm-thom}, by weakening
the restrictions on the vertices with lists of size $2$.  The following
claim is an easy consequence of Theorem~\ref{thm-thom} (see e.g. Thomassen~\cite{thom-surf},
where a slightly stronger version allowing a precolored
path of length at most $5$ is derived):

\begin{corollary}\label{cor-thm}
Let $G$ be a plane graph of girth at least $5$ and $F$ a
face of $G$.  Let $L$ be an
assignment of lists to the vertices of $G$
such that
$|L(v)|=3$ for $v\in V(G)\setminus V(F)$, $|L(v)|\ge 2$ for $v\in V(F)$, and
\begin{itemize}
\item $G$ does not contain a path $v_1v_2v_3$
with $|L(v_1)|=|L(v_2)|=|L(v_3)|=2$,
\item $G$ does not contain a path $v_1v_2v_3v_4$
with $|L(v_1)|=|L(v_2)|=|L(v_4)|=2$, and
\item $G$ does not contain a path $v_1v_2v_3v_4v_5v_6$
with $|L(v_1)|=|L(v_2)|=|L(v_5)|=|L(v_6)|=2$.
\end{itemize}
Then, $G$ can be $L$-colored.
\end{corollary}

However, even the assumptions of Corollary~\ref{cor-thm} turn out
to be unnecessarily restrictive.  The main result of the first part of our paper is
the following:

\begin{theorem}\label{thm-mainsim}
Let $G$ be a plane graph of girth at least $5$ and $F$ a
face of $G$.  Let $L$ be an
assignment of lists to the vertices of $G$
such that
$|L(v)|=3$ for $v\in V(G)\setminus V(F)$, $|L(v)|\ge 2$ for $v\in V(F)$, and
\begin{itemize}
\item $G$ does not contain a path $v_1v_2v_3$
with $|L(v_1)|=|L(v_2)|=|L(v_3)|=2$,
\item $G$ does not contain a path $v_1v_2v_3v_4v_5$
with $|L(v_1)|=|L(v_2)|=|L(v_4)|=|L(v_5)|=2$.
\end{itemize}
Then, $G$ can be $L$-colored.
\end{theorem}

The proof of this theorem is presented in Section~\ref{sec-proof}.
Let us note that the condition that $G$ does not contain a path $v_1v_2v_3v_4v_5$ with $|L(v_1)|=|L(v_2)|=|L(v_4)|=|L(v_5)|=2$
cannot be removed, as the graph in Figure~\ref{fig-cex} cannot be colored from the
prescribed lists.

\begin{figure}
\center{\epsfbox{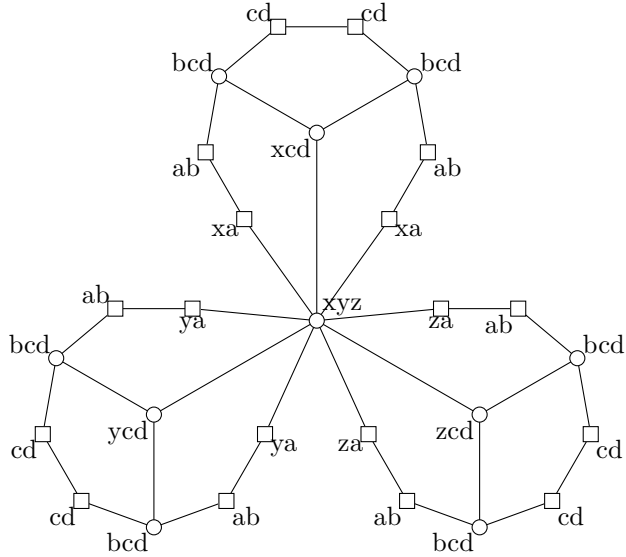}}
\caption{A counterexample for a strengthening of Theorem~\ref{thm-main}.}
\label{fig-cex}
\end{figure}

In the rest of the paper, we show two applications of Theorem~\ref{thm-mainsim}, both concerning
critical graphs.  Let us start with definitions.

A graph $G$ is {\em $k$-critical} if $G$ is not $(k-1)$-colorable, but every proper subgraph
of $G$ is $(k-1)$-colorable.  We need to generalize the notion of a critical graphs in two directions:
we need to apply it to the list coloring instead of the ordinary coloring, and we also want to consider
the situation that a subgraph of $G$ is precolored (like e.g. the path $P$ in Theorem~\ref{thm-thom}).

Consider a graph $G$, a subgraph (not necessarily induced) $S\subseteq G$ and
an assignment $L$ of lists to vertices in $V(G)\setminus V(S)$.  A graph $G$ is {\em strongly $S$-critical (with respect to $L$)}
if there exists a coloring of $S$ that does not extend to an $L$-coloring
of $G$, but extends to an $L$-coloring of every proper subgraph $G'\subset G$ such that $S\subseteq G'$.
A graph $G$ is {\em $S$-critical (with respect to $L$)} if for every proper subgraph $G'\subset G$ such that $S\subseteq G'$,
there exists a coloring of $S$ that does not extend to an $L$-coloring of $G$, but extends to an $L$-coloring of $G'$.
We call a (strongly) $S$-critical graph $G$ {\em proper} if $G\neq S$.
Note that every strongly $S$-critical graph is also $S$-critical, but the converse is false.
If $S=\emptyset$ and all vertices have the same list of $k$ colors, then $G$ is $\emptyset$-critical (or strongly $\emptyset$-critical)
if and only if $G$ is $(k+1)$-critical.

While the definition of a strongly critical graph may seem more natural, the notion of
a critical graph is often more suitable for both proofs and applications---for instance,
every graph $H\supseteq S$ has an $S$-critical subgraph $G\supseteq S$ such that each coloring
of $S$ extends to $H$ if and only if it extends to $G$ (we call such a subgraph $G$ an {\em $S$-skeleton of $H$}).
However, $H$ does not have to contain a strongly $S$-critical subgraph with this property.

In \cite{thom-surf}, Thomassen characterized the strongly $F$-critical plane graphs of girth $5$, where $F$ consists
of a boundary of a face of length at most $12$:

\begin{theorem}\label{thm-cyclesstr}
Let $G$ be a plane graph of girth at least $5$, with the outer face $F$ bounded by an induced cycle
of length at most $12$.  Let $L$ be a list assignment of lists of size three to the vertices of
$V(G)\setminus V(F)$.  If $G$ is proper strongly $F$-critical graph, then

\begin{itemize}
\item[(a)] $\ell(F)\ge 9$ and $G-V(F)$ is a tree with at most $\ell(F)-8$ vertices, or
\item[(b)] $\ell(F)\ge 10$ and $G-V(F)$ is a connected graph with at most $\ell(F)-5$ vertices
containing exactly one cycle, and the length of this cycle is $5$, or
\item[(c)] $\ell(F)=12$ and every second vertex of $F$ has degree two and is incident with a $5$-face.
\end{itemize}
\end{theorem}

Here, by $\ell(F)$ we mean the length of the face, i.e., the number of edges in its boundary walk (which for
$2$-connected graphs coincides with the number of vertices incident with $F$).  Similarly, for a path $P$,
we denote its length (number of edges) by $\ell(P)$.

If $\ell(F)\le 11$, the complete list of strongly $F$-critical graphs is provided by Theorem~\ref{thm-cyclesstr}, however for $\ell(F)=12$,
only a necessary condition is given in the case (c).  As the first application of Theorem~\ref{thm-mainsim},
we complete this classification by showing that the only $F$-critical graph satisfying the condition (c) is the one depicted
in Figure~\ref{fig-crit12}.  The proof is presented in Section~\ref{sec-cr12}.

\begin{figure}
\center{\epsfbox{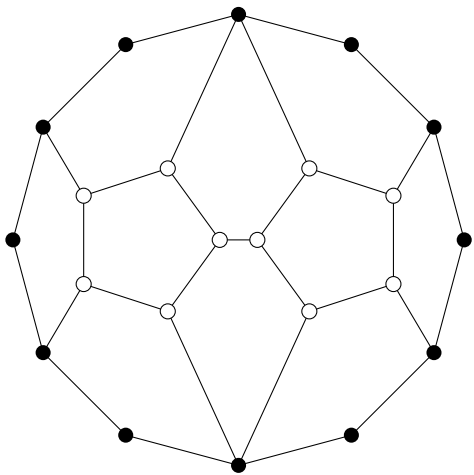}}
\caption{A critical graph bounded by a $12$-cycle.}
\label{fig-crit12}
\end{figure}

For ordinary (not list) coloring, Thomassen~\cite{thom-surf} proved that there are
only finitely many $4$-critical graphs of girth $5$ embedded in any fixed surface.  In fact,
his result allows a constant number of precolored vertices.  An alternative proof
with stronger bounds on the sizes of the critical graphs is given by Dvo\v{r}\'ak, Kr\'al' and
Thomas~\cite{proof-lincrit}.  Our goal is to
prove the same result for the list-coloring critical graphs.  We present our general argument
in a followup paper.  As the second application of Theorem~\ref{thm-mainsim},
we consider the special case of a plane graph in that vertices incident with one face
are precolored.  In Section~\ref{sec-4crit}, we show the following bound:

\begin{theorem}\label{thm-numvert}
Let $G$ be a plane graph of girth at least $5$ with the outer face $F$ bounded by a cycle of length at least $10$,
and $L$ an assignment of lists of size three to vertices of $V(G)\setminus V(F)$.  If $G$ is $F$-critical, then
$|E(G)|\le 18\ell(F)-160$ and $|V(G)|\le \frac{37\ell(F)-320}{3}$.
\end{theorem}

Let us note that this bound is much stronger than the ones shown in Thomassen~\cite{thom-surf}
(who shows that $|V(G)|\le 2^{O(\ell(F)^2)}$) or in Dvo\v{r}\'ak et al.~\cite{proof-lincrit}
(who shows that $|V(G)|\le c\ell(F)$ for a constant $c\approx 10^6$), even though these
papers only consider ordinary $3$-coloring.

\section{Proof of Theorem~\ref{thm-mainsim}}\label{sec-proof}

For the purpose of the induction, we prove an (unfortunately rather technical)
generalization of Theorem~\ref{thm-mainsim}.  In order to state this
generalization, we need to introduce several definitions.

Let $G$ be a plane graph of girth at least $5$.  Let $F$ be the outer face of $G$
and let $P=p_1\ldots p_k$ be a path with $V(P)\subseteq V(F)$.
Consider an assignment $L$ of lists to vertices of $V(G)\setminus V(P)$
such that $|L(v)|\ge 2$ for each vertex $v$ and $|L(v)|=3$ for each
$v\not\in V(F)$.  Let $I_0(G,P,L)$ be the set of vertices with the list of
size two.  Let $I(G,P,L)=I_0(G,P,L)$ if $\ell(P)\le 2$ and $I(G,P,L)=I_0(G,P,L)\cup V(P)$ otherwise.
Let us call a vertex $v$ {\em bad} if there exists a path $vv_1v_2$ with $|L(v_1)|=2$ and $v_2\in I(G,P,L)$,
or a path $vv_1v_2v_3v_4$ with $|L(v_1)|=|L(v_3)|=2$, $|L(v_2)|=3$ and $v_4\in I(G,P,L)$.
We say that the list assignment $L$ is {\em valid} if no vertex with list of size two is bad.
Let us note that if $P=\emptyset$, then $L$ is valid if and only if it satisfies the assumptions
of Theorem~\ref{thm-mainsim}.

Suppose that $\ell(P)=4$.  For a set $X\subseteq V(P)$, colorings $\psi_1$ and $\psi_2$ of $P$
are {\em $X$-different} if there exists $v\in X$ such that $\psi_1(v)\neq \psi_2(v)$.
We say that $G$ is {\em class A} if
\begin{itemize}
\item each of $p_1$ and $p_5$ is adjacent to a vertex with list of size two, and
\item there exists a coloring $\psi^{(G,P,L)}$ of $P$ such that every coloring $\psi$
of $P$ that is $\{p_1,p_2,p_4,p_5\}$-different from $\psi^{(G,P,L)}$ extends to an $L$-coloring of $G$.
\end{itemize}

We say that $G$ is {\em class B} if there exists a coloring $\psi^{(G,P,L)}$ of $P$ such that
every coloring $\psi$ of $P$ that is $\{p_1,p_3,p_5\}$-different from $\psi^{(G,P,L)}$
extends to an $L$-coloring of $G$.

\begin{theorem}\label{thm-main}
Let $G$ be a plane graph of girth at least $5$ with the outer face $F$,
let $P=p_1\ldots p_k$ be a path of length at most four such that $V(P)\subseteq V(F)$,
and let $L$ be a valid list assignment.  Furthermore, if $\ell(P)=2$, then assume
that $p_1$ or $p_3$ is not bad, and if $\ell(P)\ge 3$, then assume that no vertex of $P$ is bad.
If $G$ is a proper $P$-critical graph, then $\ell(P)=4$
and either $G$ is equal to a $5$-face, or $G$ is class A or class B.
\end{theorem}

Theorem~\ref{thm-mainsim} is the special case of Theorem~\ref{thm-main} where $P$ is empty.
Two examples of $P$-critical graphs that are class A or class B and satisfy assumptions of Theorem~\ref{thm-main}
are depicted in Figure~\ref{fig-classAB}.  Let us note that infinitely many such graphs exist.

\begin{figure}
\center{\epsfbox{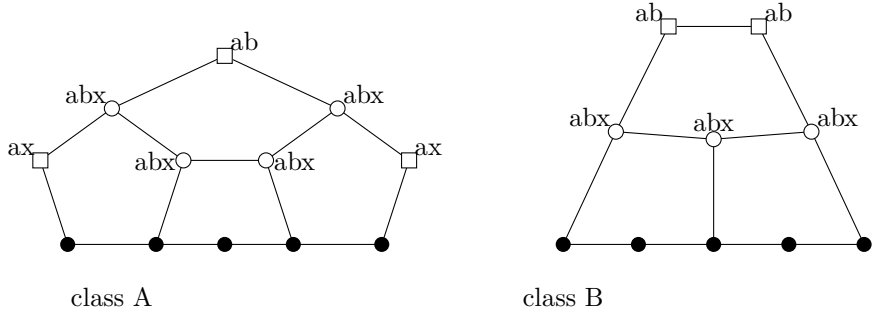}}
\caption{A class A and a class B graph.}
\label{fig-classAB}
\end{figure}

Before proving Theorem~\ref{thm-main}, let us show several observations regarding critical graphs.
Let $G$ be a $T$-critical graph (with respect to some list assignment).  For $S\subseteq G$, a
graph $G'\subseteq G$ is an {\em $S$-component} of $G$ if $S\subseteq G'$, $T\cap G'\subseteq S$ and all edges of
$G$ incident with vertices of $V(G')\setminus V(S)$ belong to $G'$.  For example, if $G$ is a plane graph with $T$
contained in the boundary of its outer face and $S$ is a cycle in $G$, then the subgraph of $G$ consisting of the vertices and edges
drawn the closed disk bounded by $S$ is an $S$-component of $G$.

\begin{lemma}\label{lemma-crs}
Let $G$ be a $T$-critical graph with list assignment $L$.  Let $G'$ be an $S$-component of $G$, for some $S\subseteq G$.
Then $G'$ is $S$-critical.
\end{lemma}
\begin{proof}
Since $G$ is $T$-critical, every isolated vertex of $G$ belongs to $T$, and thus every isolated vertex of $G'$ belongs to $S$.
Suppose for a contradiction that $G'$ is not $S$-critical.  Then, there exists an edge $e\in E(G')\setminus E(S)$ such that
every coloring of $S$ that extends to $G'-e$ also extends to $G'$.  Note that $e\not\in E(T)$.
Since $G$ is $T$-critical, there exists a coloring $\psi$ of $T$ that extends to an $L$-coloring $\varphi$ of $G-e$, but does not
extend to an $L$-coloring of $G$.  However, by the choice of $e$, the restriction of $\varphi$ to $S$ extends to an $L$-coloring
$\varphi'$ of $G'$.  Let $\varphi''$ be the coloring that matches $\varphi'$ on $V(G')$ and $\varphi$ on $V(G)\setminus V(G')$.
Observe that $\varphi''$ is an $L$-coloring of $G$ extending $\psi$, which is a contradiction.
\end{proof}

Lemma~\ref{lemma-crs} in conjunction with Theorem~\ref{thm-cyclesstr} describes the subgraphs drawn inside cycles
in plane critical graphs.  Since Theorem~\ref{thm-cyclesstr} is only stated for strongly critical graphs, 
let us show that it holds for critical graphs as well.

\begin{lemma}\label{lemma-cycles}
Let $G$ be a plane graph of girth at least $5$, with the outer face $F$ bounded by an induced cycle
of length at most $12$.  Let $L$ be a list assignment of lists of size three to the vertices of
$V(G)\setminus V(F)$.  If $G$ is proper $F$-critical graph, then $G$ satisfies one of the conditions
(a), (b) or (c) of Theorem~\ref{thm-cyclesstr}.
\end{lemma}
\begin{proof}
Suppose for a contradiction that $G$ is a counterexample to Lemma~\ref{lemma-cycles} with the smallest
number of vertices.
Since $G$ is proper, there exists a precoloring $\psi$ of $F$ that does not extend to an $L$-coloring of $G$.
Let $G'\supset F$ be the minimal subgraph of $G$ such that $\psi$ does not extend to $G'$.  Observe that
$G'$ is a proper strongly $F$-critical graph, thus it satisfies one of the conditions (a), (b) or (c).
Since $G$ does not satisfy any of these conditions, there exists an induced cycle $C\subseteq G'$ that bounds a face in $G'$,
but not in $G$.  Furthermore, if $G'$ satisfies the condition (c), then we may assume that $\ell(C)=5$.

Observe that $\ell(C)\le 8$, and $\ell(C)\le 7$ unless $\ell(F)=12$ and $|V(G')\setminus V(F)|=1$.
Let $H$ be the subgraph of $G$ drawn in the closed disk bounded by $C$.
Lemma~\ref{lemma-crs} implies that $H$ is a proper $C$-critical graph.  Since $G$ is the counterexample
to Lemma~\ref{lemma-cycles} with the smallest number of vertices, $C$ is not an induced cycle in $H$.
Since $G$ has girth at least $5$, we conclude that $\ell(C)=8$ and $C$ has a chord $e$ such that $C\cup e$
contains two $5$-cycles $C_1$ and $C_2$.  Repeating the same argument for $C_1$ and $C_2$, we conclude that
$C_1$ and $C_2$ are faces of $H$ and $V(H)=V(C)$.  It follows that $|V(G)\setminus V(F)|=1$, and thus $G$ satisfies (a).
This is a contradiction.
\end{proof}

In this section, we need only the following corollary of Lemma~\ref{lemma-cycles}.

\begin{corollary}\label{cor-cycles}
Let $G$ be a plane graph of girth at least $5$, with the outer face $F$ bounded by a cycle
of length at most $9$.  Let $L$ be a list assignment of lists of size three to the vertices of
$V(G)\setminus V(F)$.  If $G$ is a proper $F$-critical graph, then

\begin{itemize}
\item $\ell(F)\ge 8$ and $F$ has a chord, or
\item $\ell(F)=9$ and $V(G)\setminus V(F)$ consists of a single vertex $v$ adjacent
to three vertices of $F$.
\end{itemize}
\end{corollary}

Lemma~\ref{lemma-crs} together with Corollary~\ref{cor-cycles} implies that
\claim{cl-face}{if $H$ is an $S$-critical plane graph of girth at least $5$, where
$S$ is a subgraph of the boundary of the outer face of $H$, then any cycle of length at most $7$ in $H$
bounds a face, the open disk bounded by a cycle of length $8$ contains no vertices, and the open disk bounded
by a cycle of length $9$ contains at most one vertex.}

Furthermore, let us recall the following result of Vizing~\cite{vizing1976}:
\begin{theorem}\label{thm-gallai}
Let $G$ be a $2$-connected graph with a list assignment $L$ such that $|L(v)|\ge \deg(v)$ for each
vertex $v\in V(G)$.  Then $G$ is $L$-colorable, unless $G$ is a complete graph or an
odd cycle and the lists assigned to all vertices are the same.
\end{theorem}

This implies the following:
\begin{lemma}\label{lemma-gallai}
Let $G$ be a triangle-free critical graph, $S$ a subgraph of $G$ and $L$ an assignment of
lists to vertices of $V(G)\setminus V(S)$.   Let $H$ be a $2$-connected subgraph of $G$
such that $V(H)\cap V(S)=\emptyset$ and $|L(v)|\ge\deg_G(v)$ for each $v\in V(H)$.
If $G$ is $S$-critical, then $H$ is an induced odd cycle in $G$.
\end{lemma}
\begin{proof}
Let $H'=G[V(H)]$ be the subgraph of $G$ induced by $V(H)$.  Since $G$ is $S$-critical, there
exists a precoloring $\psi$ of $S$ that extends to an $L$-coloring $\varphi$ of $G-V(H)$, but not to
an $L$-coloring of $G$.  Consider the list assignment $L'$ such that for $v\in V(H)$,
$L'(v)=L(v)\setminus C_v$, where $C_v$ is the set of colors of vertices of $G-V(H)$ adjacent to $v$,
according to the coloring $\varphi$.  Observe that $H'$ is $2$-connected, $H'$ is not $L'$-colorable, and
$|L'(v)|\ge \deg_{H'}(v)$ for each $v\in V(H)$.  By Theorem~\ref{thm-gallai}, since $G$ is triangle-free,
$H'$ is an odd cycle.  Furthermore, $H=H'$, since $H$ is $2$-connected.
\end{proof}

Let us now proceed with the proof of the main result.

\begin{proof}[Proof of Theorem~\ref{thm-main}.]
Suppose that $G$ together with lists $L$ and a path $P$ is a counterexample to Theorem~\ref{thm-main}
such that $|V(G)|+|E(G)|$ is minimal, and among such graphs, the path $P$ is the longest possible.
The path $P$ is nonempty, as otherwise we can choose an arbitrary vertex of $F$ as $p_1$.
As $G$ is a proper $P$-critical graph, there exists at least one precoloring of $P$ that does not extend to an $L$-coloring of $G$.
By the minimality of $G$, each vertex of $P$ has degree at least two.
By Lemma~\ref{lemma-crs}, each vertex $v\in V(G)\setminus V(P)$ has degree at least $|L(v)|$.

\begin{lemma}\label{lemma-conn}
The graph $G$ is $2$-connected.
\end{lemma}
\begin{proof}
Obviously, $G$ is connected.  Suppose now that $v$ is a cut vertex of $G$ and $G_1$ and $G_2$
are induced subgraphs of $G$ such that $G=G_1\cup G_2$, $\{v\}=V(G_1)\cap V(G_2)$ and $|V(G_1)|,|V(G_2)|\ge 2$.
Let $P_i=P\cap G_i$ if $v\in V(P)$ and $P_i=v$ otherwise, for $i\in \{1,2\}$; by Lemma~\ref{lemma-crs},
$G_i$ is $P_i$-critical.  By symmetry, we may assume that $\ell(P_1)\le \ell(P_2)$, and thus $\ell(P_1)\le \ell(P)/2\le 2$.
It follows that $v\not\in I(G_1,P_1,L)$ and $I(G_1,P_1,L)\subseteq I(G,P,L)$, thus the restriction of $L$ to $V(G_1)\setminus V(P_1)$
is a valid list assignment.  If $\ell(P_1)=2$, with say $P_1=p_1p_2p_3$ and $p_3=v$, then $p_1$ is not bad in $G_1$, since it is not
bad in $G$.  By the minimality of $G$, we can apply Theorem~\ref{thm-main} to $G_1$, obtaining $G_1=P_1$.
Since $|V(G_1)|\ge 2$, we conclude that $G$ contains a vertex of degree one, which is a contradiction.
\end{proof}

By Lemma~\ref{lemma-conn}, the outer face $F$ of $G$ is bounded by a cycle.  A {\em chord} of $F$ is an edge in $E(G)\setminus E(F)$
incident with two vertices of $V(F)$.  A {\em $t$-chord} of $F$ is a path $Q=q_0q_1\ldots q_t$ of length $t$ ($t\ge 2$) such
that $q_0\neq q_t$ and $V(Q)\cap V(F)=\{q_0,q_t\}$.  Sometimes, we refer to a chord as a {\em $1$-chord}.

\begin{lemma}\label{lemma-1chord}
The cycle $F$ has no chords.
\end{lemma}
\begin{proof}
Suppose that $e=uv$ is a chord of $F$, and let $G_1$ and $G_2$ be the two induced subgraphs of $G$ such
that $G=G_1\cup G_2$, $uv=G_1\cap G_2$ and $G_1,G_2\neq uv$.  Note that $|V(G_1)|,|V(G_2)|>2$.
If $P\subseteq G_1$, then $G_2$ is $uv$-critical by Lemma~\ref{lemma-crs}.
Since $I(G_2,uv,L)\subseteq I(G,P,L)$, the restriction of $L$ to $G_2$ is a valid list assignment.
By the minimality of $G$, we have $G_2=uv$, which is a contradiction.  It follows that $P\not\subseteq G_1$
and by symmetry, $P\not\subseteq G_2$.  Therefore, every chord of $F$ is incident with a vertex of $P$ distinct
from $p_1$ and $p_k$.

Suppose now that say $|(V(P)\cap V(G_1))\setminus \{u,v\}|\le 1$.  In that case, $P_1=(P\cap G_1)+uv$ has length
at most two.  By Lemma~\ref{lemma-crs}, $G_1$ is $P_1$-critical, and since $I(G_1,P_1,L)\subseteq I(G,P,L)$,
we conclude that the restriction of $L$ to $G_1$ is a valid list assignment.  Furthermore, if $\ell(P_1)=2$,
then we may assume that $P_1=p_1uv$, and $p_1$ is not bad in $G_1$ since it is not bad in $G$.
By the minimality of $G$, we conclude that $G_1=P_1$, which is a contradiction, since $|V(G_1)|>2$
and $G$ does not contain a vertex of degree one.

By symmetry, we conclude that $|V(P)\cap V(G_i)\setminus \{u,v\}|\ge 2$ for $i\in\{1,2\}$.
This implies that $k=5$ and $V(P)\cap \{u,v\}=p_3$.  Without loss of generality, $u=p_3$,
$P\cap G_1=p_1p_2p_3$ and $P\cap G_2=p_3p_4p_5$.
Let $v_i$ be the neighbor of $v$ in the facial walk of $F$ belonging to $G_i$, for $i\in\{1,2\}$.
Since every chord of $F$ is incident with a vertex of $P$,
$v$ is not adjacent to a vertex with list of size two except for $v_1$ and $v_2$.

Suppose now that $|L(v_1)|\neq 2$.  By Lemma~\ref{lemma-crs}, $G_1$ is $P'_1$-critical,
where $P'_1=p_1p_2p_3v$. Note that $I(G_1,P_1,L)\setminus I(G,P,L)=\{v\}$, and
$v$ is not bad, as it is not adjacent to a vertex with list of size two in $G_1$.
By the minimality of $G$, we conclude that $G_1=P_1$, which is a contradiction,
since $p_1$ has degree at least two in $G$.

By symmetry, $|L(v_1)|=|L(v_2)|=2$.  Since $L$ is a valid list assignment, $|L(v)|=3$.
By Lemma~\ref{lemma-crs}, $G_1$ is $P_1$-critical and $G_2$ is $P_2$-critical,
where $P_1=p_1p_2p_3vv_1$ and $P_2=p_5p_4p_3vv_2$.  
Note that both $G_1$ and $G_2$ are proper, since $p_1$ and $p_5$ have degree at least two in $G$.
Let $L_i$ be $L$ restricted to $V(G_i)\setminus V(P_i)$, for $i\in \{1,2\}$.  Then $I(G_i,P_i, L_i)\setminus I(G,P,L)=\{v\}$,
and since $v$ is not adjacent to a vertex with list of size two in $G_i$, $v$ is not bad.

By the minimality of $G$, Theorem~\ref{thm-main} holds for $G_1$ and $G_2$.
Since $L$ is a valid list assignment for $G$, by the symmetry between $v_1$ and $v_2$
we may assume that $v_1$ is not adjacent to a vertex of $I(G,P,L)$, and thus
$G_1$ is neither a $5$-face nor class A.  Therefore, $G_1$ is class B.
Let $\psi_1=\psi^{(G_1,P_1,L)}$.  Consider a precoloring $\psi$ of $P$ that is $\{p_1,p_3\}$-different
from $\psi_1$.  By the minimality of $G$, the precoloring of the path $p_3p_4p_5$
given by $\psi$ extends to an $L$-coloring $\varphi_2$ of $G_2$.
The precoloring $\psi'$ of $P_1$ given by $\psi'(p_i)=\psi(p_i)$ for $i\in\{1,2,3\}$,
$\psi'(v)=\varphi_2(v)$ and $\psi'(v_1)\in L(v_1)\setminus \{\varphi(v)\}$ extends to an $L$-coloring $\varphi_1$ of $G_1$,
since $G_1$ is class B and $\psi'$ is $\{p_1,p_3\}$-different from $\psi_1$.  We conclude that
$\psi$ extends to the $L$-coloring $\varphi_1\cup \varphi_2$ of $G$.

Suppose that $G_2$ is class A or B, with $\psi_2=\psi^{(G_2,P_2,L)}$.  Analogically to the previous paragraph, we conclude that
any precoloring $\psi$ of $P$ that is $\{p_5\}$-different from $\psi_2$ extends to
an $L$-coloring of $G$.  It follows that $G$ is class B with $\psi^{(G,P,L)}=\psi_0$, where
$\psi_0(p_1)=\psi_1(p_1)$, $\psi_0(p_3)=\psi_1(p_3)$ and $\psi_0(p_5)=\psi_2(p_5)$ ($\psi_0(p_2)$ and $\psi_0(p_4)$
are arbitrary colors distinct from the colors used on the rest of $P$).  This is a contradiction,
since $G$ is a counterexample.

Since $G_2$ satisfies the conclusion of Theorem~\ref{thm-main}, $G_2$ is a $5$-face and $v_2p_5$ is an edge.
Choose $c'\in L(v_1)\setminus \{\psi_1(v_1)\}$, $c\in L(v)\setminus\{c',\psi_1(p_3)\}$ and $d\in L(v_2)\setminus \{c\}$.
Then, $G$ is class B with $\psi^{(G,P,L)}=\psi_0$, where $\psi_0(p_1)=\psi_1(p_1)$, $\psi_0(p_3)=\psi_1(p_3)$ and $\psi_0(p_5)=d$:
before, we proved that if a precoloring $\psi$ of $P$ is $\{p_1,p_3\}$-different from $\psi_0$, then $\psi$ extends to an $L$-coloring of $G$.
If $\psi(p_3)=\psi_1(p_3)$ and $\psi(p_5)\neq d$, then we can color $v$ by $c$, $v_2$ by $d$ and $v_1$ by $c'$.  The resulting
coloring of $P_1$ is $\{v_1\}$-different from $\psi_1$, thus it extends to $G_1$, giving an $L$-coloring of $G$.
This is a contradiction, since $G$ is a counterexample to Theorem~\ref{thm-main}.
\end{proof}

By Lemma~\ref{lemma-1chord}, $P$ is a subpath of $F$ and by Corollary~\ref{cor-cycles}, $\ell(F)\ge 9$.
Also, $\ell(P)\ge 3$:
\begin{itemize}
\item If $P$ consists of a single vertex $p_1$, then we can add arbitrary neighbor of $p_1$ in $F$ as $p_2$.
Since $p_1p_2$ is longer than $P$ and the assumptions of Theorem~\ref{thm-main} are satisfied,
the choice of $P$ implies that $G=p_1p_2$, which is a contradiction.
\item If $P=p_1p_2$ and $p_2$ is not bad, then let $P'=xp_1p_2$, where $x$ is the neighbor of $p_1$ in $F$.
Otherwise, let $P'=p_1p_2y$, where $y$ is the neighbor of $p_2$ in $F$.  As $p_2$ is bad with respect to $P$, it
follows that $|L(y)|=2$, and since $L$ is a valid assignment, $y$ is not bad with respect to $P'$.
Therefore, the assumptions of Theorem~\ref{thm-main}
are satisfied.  Again, we conclude that $G=P'$, which is a contradiction.
\item Suppose that $\ell(P)=2$ and say $p_1$ is not bad.  Let $x\neq p_2$ be the neighbor of $p_3$ in $F$.
If $|L(x)|=2$, then let $P'=p_1p_2p_3x$.  Otherwise, $p_3$ is not bad, and by the symmetrical argument we can assume that no neighbor of $p_1$
has list of size two, either.  Let $y\neq p_3$ be the neighbor of $x$ in $F$.  If $|L(y)|=2$, then let
$P'=p_1p_2p_3xy$, otherwise let $P'=p_1p_2p_3x$.  Let $L'$ be $L$ restricted to $V(G)\setminus V(P')$.

Observe that $I(G,P',L')\setminus I(G,P,L)\subseteq \{p_1,p_2,p_3,x\}$.  We conclude that if no vertex of $P'$ is
bad, then $L'$ is a valid assignment.  The vertices $p_2$ and $p_3$ are not adjacent to a vertex with list of
size two, thus they are not bad.  If $|L(x)|=2$, then neither $p_1$ nor $x$ are bad with respect to $L$,
and we conclude that neither of them is bad with respect to $L'$.  If $|L(x)|=3$, then
neither $p_1$ nor $x$ are adjacent to a vertex with list of size two in $L'$, thus they are not bad.
Finally, if $y\in V(P')$, then $y$ is not bad with respect to $L'$, since it is not bad with respect to
$L$ and no other vertex of $P'$ is bad.  We conclude that the assumptions of Theorem~\ref{thm-main}
are satisfied, and since $P'$ is longer than $P$, $G$ satisfies the conclusions of Theorem~\ref{thm-main}
with respect to $P'$.  Since the minimum degree of $G$ is at least two, we have $G\neq P'$,
and thus $\ell(P')=4$.  Note that $G$ is not a $5$-face, as $|L(x)|=3$ and $x$ would have degree two.
It follows that $G$ is class A or B.  Let $\psi_0=\psi^{(G,P',L')}$.

For any precoloring $\psi$ of $P=p_1p_2p_3$, we can choose a color $c\in L(y)\setminus \{\psi_0(y)\}$, color
$y$ with $c$ and $x$ by a color in $L(x)\setminus \{\psi(p_3),c\}$, and extend this precoloring (which is $\{y\}$-different
from $\psi_0$) to an $L$-coloring of $G$.  This shows that $G$ cannot be a proper $P$-critical graph, which is a contradiction.
\end{itemize}

Let $D=\{v\in V(G)\setminus V(P) : |L(v)|=2\}\cup\{p_1,p_k\}$.  In following lemmas, we show that short chords are not incident with
vertices of $D$.  Let us first prove a simpler version of this claim.

\begin{lemma}\label{lemma-dchords}
Let $Q=q_0\ldots q_t$ be a $t$-chord of $F$ ($t\le 3$) and let $G_1,G_2\neq Q$ be the subgraphs of $G$ such that $G=G_1\cup G_2$ and $G_1\cap G_2=Q$.
If all vertices of $P$ except for $p_1$ and $p_k$ belong to $V(G_1)$, then $|\{q_0,q_t\}\cap D|\le t-2$.
\end{lemma}
\begin{proof}
Suppose for a contradiction that $Q$ does not satisfy the conclusions of the lemma, i.e., $|\{q_0,q_t\}\cap D|\ge t-1$.
Subject to these assumptions, choose $Q$ such that $G_2$ is as small as possible.
By Lemma~\ref{lemma-1chord}, we have $t\ge 2$.
Let $Q'$ be the path obtained from $Q$ in the following way:  for $i\in\{0,t\}$,
\begin{itemize}
\item if $q_i=p_2$, then add $p_1$ to $Q'$,
\item if $q_i=p_{k-1}$, then add $p_k$ to $Q'$, and
\item if $q_i\not\in D$ and $q_i$ is adjacent to a vertex $v\in V(F)\cap V(G_2)$ with $|L(v)|=2$,
then add $v$ to $Q'$.
\end{itemize}
Let $Q'=q'_0q'_1\ldots q'_{\ell(Q')}$ and let $L_2$ be $L$ restricted to $V(G_2)\setminus V(Q')$.
Note that $Q'\subseteq G_2$: otherwise say $p_1\in V(G_1)$
and $p_2=q_0$, implying that $p_3\in V(G_2)$.  Since $p_3\in V(G_1)$ by the assumptions of the lemma,
it follows that $p_3=q_t$, and the cycle $q_0q_1\ldots q_t$ of length $t+1$ contradicts the assumption that the girth of $G$ is at least $5$.

Observe that $\ell(Q')\le 3$.  By Lemma~\ref{lemma-1chord} and the minimality of $G_2$, $q'_i$ is not adjacent to a vertex with list of size two
in $L_2$, for $1\le i\le \ell(Q')-1$.  Also, if $x\in\{q'_0,q'_{\ell(Q')}\}$ is adjacent to a vertex with list of size two in $L_2$,
then $x\in I(G,P,L)$.  It follows that $L_2$ is a valid list assignment for $G_2$ with respect to $Q'$
and no vertex of $Q'$ is bad in $G_2$.  By Lemma~\ref{lemma-crs} and the minimality of $G$,
it follows that $G_2=Q'$, which is a contradiction.
\end{proof}

\begin{lemma}\label{lemma-chords}
Let $Q=q_0\ldots q_t$ be a $t$-chord of $F$ ($t\le 4$) and let $G_1,G_2\neq Q$ be the subgraphs of $G$ such that $G=G_1\cup G_2$ and $G_1\cap G_2=Q$.
If all vertices of $P$ except for $p_1$ and $p_k$ belong to $V(G_1)$
and $G_2$ does not consist of a single $5$-face, then
$t\ge 3$ and $|\{q_0,q_t\}\cap D|\le t-3$.
\end{lemma}
\begin{proof}
Suppose for a contradiction that $Q$ does not satisfy the conclusions of the lemma, i.e., $|\{q_0,q_t\}\cap D|\ge t-2$,
and subject to this choose $Q$ so that $G_2$ is as large as possible.  Let $Q'$ be the path obtained from $Q$ in the 
same way as in the proof of Lemma~\ref{lemma-dchords}.  Again, $Q'\subseteq G_2$, as otherwise
the cycle $q_0q_1\ldots q_t$ is a cycle of length $t+1$ forming the outer face of $G_2$; however by (\ref{cl-face}),
it would follow that $G_2$ is a $5$-face, contradicting the assumptions.

Note that $\ell(Q')\le 4$.
By Lemmas~\ref{lemma-1chord} and \ref{lemma-dchords}, if a vertex $x\in V(Q')$ is adjacent to a vertex with
list of size two in $L_2$, then $x\in \{q'_0,q'_{\ell(Q')}\}\cap I(G,P,L)$.
We conclude that $L_2$ is a valid list assignment for $G_2$ and that no vertex of $Q'$ is bad in $G_2$.
Lemma~\ref{lemma-crs} implies that $G_2$ is a proper $Q'$-critical graph.
By the minimality of $G$, $\ell(Q')=4$ and $G_2$ is class A or B.
Let $\psi_2=\psi^{(G_2,Q', L_2)}$.  Note that $q'_2\not\in V(F)$ and $q'_0,q'_4\in D$.

If $G_2$ is class A, then let $G'$ consist of $G_1$ together with two new vertices $x$ and $y$
and a path $q'_1xyq'_3$, with the list assignment $L'$ given by $L'(v)=L(v)$ for $v\in V(G_1)\setminus V(P)$,
$L'(x)=\{\psi_2(q'_1),c\}$ and $L'(y)=\{\psi_2(q'_3),c\}$, where $c$ is an arbitrary color distinct from
$\psi_2(q'_1)$ and $\psi_2(q'_3)$.  If $G_2$ is class B, then let $G'=G_1$,
$L'(v)=L(v)$ for $v\in V(G_1)\setminus \{q'_2\}$ and $L'(q'_2)=L(q'_2)\setminus \{\psi_2(q'_2)\}$.
Consider any precoloring $\psi$ of $P$ whose restriction to $P\cap G'$ extends to an $L'$-coloring $\varphi$ of $G'$,
and let $\varphi'$ be the restriction of $\psi\cup \varphi$ to $V(Q')$.  If $G_2$ is class A, then
$\varphi'$ is $\{q'_1,q'_3\}$-different from $\psi_2$, and if $G_2$ is class B, then $\varphi'$ is
$\{q'_2\}$-different from $\psi_2$, thus $\varphi'$ extends to an $L_2$-coloring of $G_2$.  Together
with $\varphi$, this gives an $L$-coloring of $G$ extending $\psi$.  Since at least one precoloring of
$P$ does not extend to an $L$-coloring of $G$, we conclude that there exists a precoloring of
$P\cap G'$ that does not extend to an $L'$-coloring of $G'$.  Therefore, a $(P\cap G')$-skeleton $G''$ of $G'$
is a proper $(P\cap G')$-critical graph.  Suppose now that
\claim{ass-valid}{$L'$ is a valid list assignment and no vertex of $P\cap G''$ is bad.}
In order to apply Theorem~\ref{thm-main}, we need to show that
$G''$ is smaller than $G$, i.e., that $|V(G'')|+|E(G'')|\le |V(G_1)|+|E(G_1)|+5 < |V(G)|+|E(G)|$.
This is obvious if $|V(G_2)\setminus V(Q)|\ge 3$.  Since $G_2$ is not a $5$-face, we have $|V(G_2)\setminus V(Q)|\ge 1$
and $t\ge 5-|V(G_2)\setminus V(Q)|$.  If $|V(G_2)\setminus V(Q)|=1$, then $t=4$, $q_0,q_4\in D$ and the vertex
$w\in V(G_2)\setminus V(Q)$ has degree two.  It follows that $|L(w)|=2$, and the path $q_0wq_4$ contradicts
the assumptions of Theorem~\ref{thm-main}.  Similarly, we exclude the case that $|V(G_2)\setminus V(Q)|=2$.

Note that if $\ell(P\cap G'')=4$, then $G''$ does not consist of a single $5$-face, since $F$ does not have chords.
By Theorem~\ref{thm-main} applied to $G''$ with the path $P\cap G'$ and the list assignment $L'$,
we have $\ell(P)=4$, $P\subseteq G'$ and $G''$ is class A or B.  Let $\psi_1=\psi^{(G'',P,L')}$.

If $G''$ is class A, then any precoloring of $P$ that is $\{p_1,p_2,p_4,p_5\}$-different from $\psi_1$
extends to an $L'$-coloring of $G''$, and thus it also extends to an $L'$-coloring of $G'$ and an $L$-coloring
of $G$.  Also, $p_1$ is adjacent to a vertex $w$ such that $|L'(w)|=2$ in $G''$.  Note that $p_1\not\in \{q'_1,q'_3\}$
and by Lemma~\ref{lemma-dchords}, $p_1$ is not adjacent to $q'_2$, thus $w\in V(G)$ and $|L(w)|=2$.  By symmetry,
$p_k$ has a neighbor with list of size two in $G$.  Therefore, $G$ is class A.  Similarly, if $G''$ is class B,
then $G$ is class $B$.  This is a contradiction, and hence the assumption (\ref{ass-valid}) is false.

Let us now distinguish the two cases regarding whether $G_2$ is class A or B with respect to the path
$Q'$ and the list assignment $L_2$:

\begin{itemize}
\item {\em $G_2$ is class A.} Let $z$ be the neighbor of $q'_0$ in $G_2$ with $|L(z)|=2$.

Let us recall that in this case the list assignment $L'$
matches $L$ on $V(G'')\setminus V(P)$ and $G'$ contains two additional vertices $x$ and $y$
with lists of size two.  Let us call a vertex $q'\in\{q'_1,q'_3\}$ {\em problematic} if
either $q'\in I(G'',G'\cap P, L')$, or $|L'(q')|=3$ and $G''$ contains a path
$q'uv$ with $|L'(u)|=2$ and $v\in (I(G'',G'\cap P, L')\setminus\{x,y\})\subseteq I(G,P,L)$.
Since (\ref{ass-valid}) is false, we may assume by symmetry that $q'_1$ is problematic.

By the choice of $Q'$, either $q'_1\in V(P)$ or $|L(q'_1)|=3$.  Suppose first that $q'_1\not\in V(P)$.
If $q'_1\in V(F)$, then $|L(q'_0)|=2$ by the construction of $Q'$, and the path
$zq'_0q'_1uv$ contradicts the assumption that $L$ is a valid list assignment.
If $q'_1\not\in V(F)$, then by Lemma~\ref{lemma-dchords} $uq'_1q'_0$ is not a $2$-chord of $F$,
and thus $u=q'_0$.  In this case, the path $zuv$ contradicts the assumption that $L$ is a valid list assignment.
Therefore, $q'_1\in V(P)$.  The symmetric argument shows that if $q'_3$ is problematic, then $q'_3\in V(P)$.

By symmetry and the construction of $Q'$, we may assume that
$q'_1=p_2$ and $q'_0=p_1$.  Note that $q'_3\not\in V(P)$, as the girth of $G$ is at least five and $q'_3\neq p_5$.
It follows that $q'_3$ is not problematic, and thus $L'$ is a valid list assignment for $G''$
with respect to the path $P'=xp_2\ldots p_k$ and no vertex of this path is bad.

By the minimality of $G$, this implies that $\ell(P')=\ell(P)=4$ and $G''$ is class A or B with respect to $P'$,
with $\psi_0=\psi^{(G'',P', L')}$.  If $G''$ is class A with respect to $P'$,
then $p_5$ is adjacent to a vertex $w$ with $|L(w)|=|L'(w)|=2$.  Furthermore,
$p_1=q'_0$ is adjacent to $z$, which has $|L(z)|=2$.  Let $\psi^{(G,P,L)}$ be a coloring
that matches $\psi_0$ on $p_2p_3p_4p_5$ and satisfies $\psi_2(p_1)\in \{\psi^{(G,P,L)}(p_1),\psi^{(G,P,L)}(p_2)\}$.
Consider a precoloring $\psi$ of $P$.  If $\psi$ is $\{p_2,p_4,p_5\}$-different from $\psi^{(G,P,L)}$, then
$\psi$ is $\{p_2,p_4,p_5\}$-different from $\psi_0$; choose a color of $x$ in $L'(x)\setminus\psi(p_2)$
and extend the resulting precoloring of $P'$ to an $L'$-coloring of $G''$.  This implies that $\psi$ extends
to an $L$-coloring of $G$.  If $\psi$ is not $\{p_2,p_4,p_5\}$-different from $\psi^{(G,P,L)}$, but it is $\{p_1\}$-different, then
$\psi(p_1)\neq\psi_2(p_1)$.  In this case, by Theorem~\ref{thm-main} applied to a $p_2p_3p_4p_5$-skeleton of $G_1$
with list assignment $L$, we conclude that $\psi$ extends to an $L$-coloring $\varphi$ of $G_1$, and since $\varphi\cup\psi$
is $\{p_1\}$-different from $\psi_2$ on $Q'$, $\varphi$ extends to $L_2$-coloring of $G_2$, giving an $L$-coloring of $G$.
We conclude that $G$ is class A.  Analogically, if $G''$ is class B with respect to $P'$, then
$G$ is class B.  This is a contradiction.

\item {\em $G_2$ is class B.}
The vertex $q'_2$ does not have any neighbor in $D$ by Lemma~\ref{lemma-dchords}.
Since (\ref{ass-valid}) is false, $q'_2$ has a neighbor $p\in V(P)\setminus \{p_1,p_k\}$.
As girth of $G$ is at least five, $q'_2$ is adjacent to exactly one vertex of $P$.
Since (\ref{ass-valid}) is false, $G_1$ contains a path $q'_2uv$ with $|L(u)|=3$ and $|L(v)|=2$.
Since $G_2$ was chosen to be as large as possible, we may assume that $u=q'_3$,
and if $q'_4\in V(G_1)$, then $v=q'_4$.

If $\ell(P)=4$ and $q'_2$ is adjacent to $p_3$, then consider precoloring $\psi$ of $P$
that does not extend to an $L$-coloring of $G$.  Choose a color for $q'_2$ from
$L'(q'_2)\setminus\{\psi(p_3)\}$.  Let $H_1,H_2\neq q'_2p_3$ be the subgraphs of $G_1\cup P$ such that
$G_1\cup P=H_1\cup H_2$, $q'_2p_3=H_1\cap H_2$ and $p_1\in V(H_1)$.  By the minimality of $G$,
Theorem~\ref{thm-main} implies that the precoloring of $p_1p_2p_3q'_2$ extends to an $L$-coloring of $H_1$
and the precoloring of $p_5p_4p_3q'_2$ extends to an $L$-coloring of $H_2$, giving an $L'$-coloring of $G'$.
This implies that $\psi$ extends to an $L$-coloring of $G$, which is a contradiction.

We conclude that $p\in \{p_2,p_{k-1}\}$, and by symmetry, we may assume that
$p=p_2$.  The maximality of $G_2$ implies that $q'_1=p_2$ and $q'_0=p_1$.
Note that $L'$ is a valid list assignment with respect to the path $P'=q'_2p_2\ldots p_k$, and no vertex of this
path is bad.  By the minimality of $G$, $\ell(P')=\ell(P)=4$ and $G''$ is class A or B with respect to the path $P'$.
Since $q'_2$ is not adjacent to a vertex with list of size two, we conclude that $G''$ is class B.
It follows that $G$ is class B, with $\psi^{(G,P,L)}$ matching $\psi^{(G'',P',L')}$ on $p_3$ and $p_5$
and $\psi^{(G,P,L)}(p_1)=\psi_2(p_1)$.  This is again a contradiction, finishing the proof of the lemma.
\end{itemize}
\end{proof}

Let $P'=p_1\ldots p_kv_1v_2v_3v_4v_5$ be a subpath of $F$.  As we observed before, $\ell(F)\ge 9$.
Suppose that $k=5$ and $\ell(F)=9$, i.e., $v_5=p_1$.
By Corollary~\ref{cor-cycles}, $G$ contains exactly one vertex $v\not\in V(F)$.
As $p_1$ and $p_5$ are not bad, $v$ must be adjacent to $p_3$, $v_1$ and $v_4$,
$|L(v_1)|=|L(v_4)|=3$ and $|L(v_2)|=|L(v_3)|=2$, i.e., $G$ is the class B graph depicted in Figure~\ref{fig-classAB}
($L$ may differ from the list assignment shown in the figure).  Therefore, we may assume that all the vertices of $P'$ are distinct.

\begin{lemma}\label{lemma-nott}
Exactly one of $|L(v_1)|=2$ and $|L(v_2)|=2$ is satisfied.  Furthermore, if
$\psi$ is a precoloring of $P$ that cannot be extended to an $L$-coloring of $G$,
then $\psi(p_k)\in L(v_1)$.
\end{lemma}
\begin{proof}
Since $p_k$ is not bad, it cannot be the case that $|L(v_1)|=|L(v_2)|=2$.
Let $\psi$ be a precoloring of $P$ that does not extend to an $L$-coloring of $G$.
Suppose that $|L(v_1)|=|L(v_2)|=3$ or $\psi(p_k)\not\in L(v_1)$.
Let $N'$ be the set of neighbors of $p_k$ in $G$.  Let $N=N'\setminus\{p_{k-1},v_1\}$ if
$\psi(p_k)\not\in L(v_1)$ and $N=N'\setminus\{p_{k-1}\}$ otherwise.
Let $L'$ be the list assignment obtained from $L$ by removing $\psi(p_k)$ from the
lists of all vertices in $N$.
The vertices of $N$ form an independent set in $G$.  By Lemma~\ref{lemma-chords} and the assumption that $|L(v_2)|=3$,
if $w$ is a neighbor of a vertex of $N$ and $w\not\in V(P)$, then $|L(w)|=3$.  Similarly, if $w\in V(P)$, then
$w\not\in\{p_1,p_2\}$, and since the girth of $G$ is at least $5$, $w\not\in \{p_3,\ldots, p_{k-1}\}$.
Therefore, $L'$ is a valid list assignment for $G-p_k$ with respect to the path $P-p_k$ and no vertex of $P-p_k$ is bad.
By the minimality of $G$, we can apply Theorem~\ref{thm-main} to a $(P-p_k)$-skeleton of $G-p_k$, and conclude
that $\psi$ can be extended to an $L'$-coloring of $G-p_k$.  Therefore, $\psi$ extends to an $L$-coloring of $G$,
which is a contradiction.
\end{proof}

Let us define a set $X$ of vertices of $G$, depending on the sizes of the lists of vertices $v_1$, \ldots, $v_5$
(we exclude the cases forbidden by Lemma~\ref{lemma-nott} and the assumption that $p_k$ is not bad).  See Figure~\ref{fig-defX}
for an illustration.
\begin{itemize}
\item If $|L(v_1)|=2$, then $|L(v_2)|=3$.  If $|L(v_3)|=3$, then let
$X=\{v_1\}$.
\item If $|L(v_1)|=|L(v_3)|=2$ and $|L(v_2)|=3$, then $|L(v_4)|=3$.
If $|L(v_5)|=3$, then let $X=\{v_1,v_2,v_3\}$, otherwise
let $X=\{v_1,v_2,v_3, v_4\}$.
\item If $|L(v_1)|=3$, then $|L(v_2)|=2$.  If $|L(v_3)|=|L(v_4)|=3$,
then let $X=\{v_1,v_2\}$.
\item If $|L(v_1)|=|L(v_3)|=3$ and $|L(v_2)|=|L(v_4)|=2$, then let
$X=\{v_1,v_2,v_3\}$.
\item If $|L(v_1)|=3$ and $|L(v_2)|=|L(v_3)|=2$, then $|L(v_4)|=3$.
If $|L(v_5)|=3$, then let $X=\{v_1,v_2,v_3\}$, otherwise
let $X=\{v_1,v_2,v_3, v_4\}$.
\end{itemize}
\begin{figure}
\center{\epsfbox{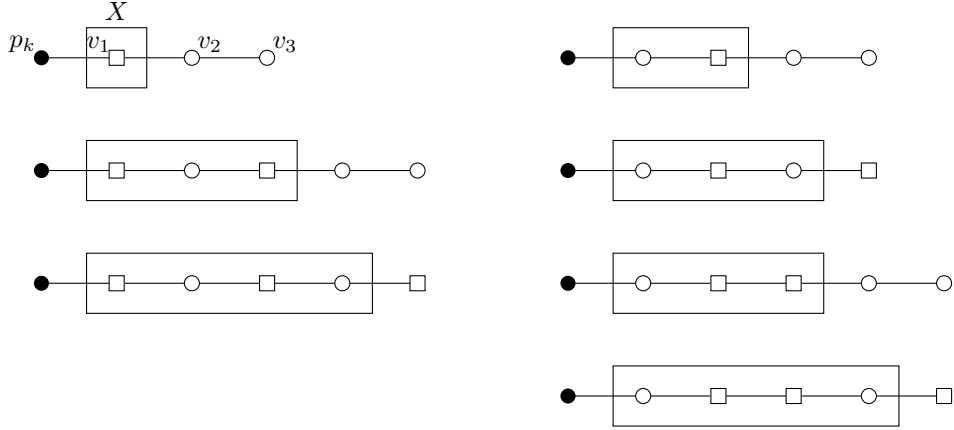}}
\caption{The definition of the set $X$.  Squares denote vertices with list of size two.}
\label{fig-defX}
\end{figure}

Let $m=|X|$.  Let us fix a precoloring $\psi$ of $P$ that does not extend to an $L$-coloring of $G$.
Observe that there exists an $L$-coloring $\varphi=\varphi_{\psi}$ of the path induced by $X$ such that
\begin{itemize}
\item $\varphi(v_1)\neq\psi(p_k)$, and
\item if $|L(v_{m+1})|=2$, then $\varphi(v_m)\not\in L(v_{m+1})$.
\end{itemize}
Furthermore, if $|L(v_{m+1})|=3$, then $v_m$ is the only neighbor of $v_{m+1}$
that belongs to $I(G,P,L)$.

Let $X'=X\cup \{v\in \{v_{m+1},p_k\} : \deg_G(v)=2\}$ and $G'=G-X'$.
Let $N'$ be the set of neighbors of the vertices of $X$ in $V(G)\setminus (X'\cup\{p_k\})$.
Let $N=N'$ if $|L(v_{m+1})|=3$ and $N=N'\setminus \{v_{m+1}\}$ if $|L(v_{m+1})|=2$.

Let $L'$ be the assignment of lists to vertices of $G'$ obtained from $L$ by removing the colors of vertices of $X$ given by $\varphi$
from the lists of their neighbors, i.e., from the lists of vertices in $N$ (or, more precisely, $N'$; however, when
$N\neq N'$, then $N'\setminus N=\{v_{m+1}\}$, the only neighbor of $v_{m+1}$ in $X$ is $v_m$ and $\varphi(v_m)\not\in L(v_{m+1})$).
Additionally, if $p_k\not\in X'$, then we set $L'(p_k)=\{\psi(p_{k-1}),\psi(p_k)\}$,
and if $\ell(P)=4$, then $L'(p_1)=\{\psi(p_1),\psi(p_2)\}$.
Let $P'=p_2p_3p_4$ if $\ell(P)=4$ and $P'=p_1p_2p_3$ if $\ell(P)=3$.
Since $\psi$ does not extend to an $L$-coloring of $G$,
we conclude that $\psi$ (restricted to the path $P'$)
does not extend to an $L'$-coloring of $G'$, either.
Let us remark that $\ell(P')=2$, thus the vertices of $P'$ do not belong to $I(G',P',L')$,
and we only need to show that the list assignment is valid and $p_{k-1}$ is not bad in order to
be able to apply Theorem~\ref{thm-main}.

First, assume that $G$ does not contain the following configurations, see Figure~\ref{fig-obs} for
an illustration:
\begin{description}
\item[Obstruction A.] A path $p_{k-1}xy$ with $x,y\in N$.
\item[Obstruction B.] A path $vxv_{m+1}$, where $v\in X$ and $x\in N$.
\item[Obstruction C.] A path $xyz$ with $x$ adjacent to $p_k$, $y$ to $v_2$ and $z$ to $v_4$,
in case that $v_4\in X$.
\item[Obstruction D.] A vertex in $N$ with two neighbors in $X$.
\end{description}
\begin{figure}
\center{\epsfbox{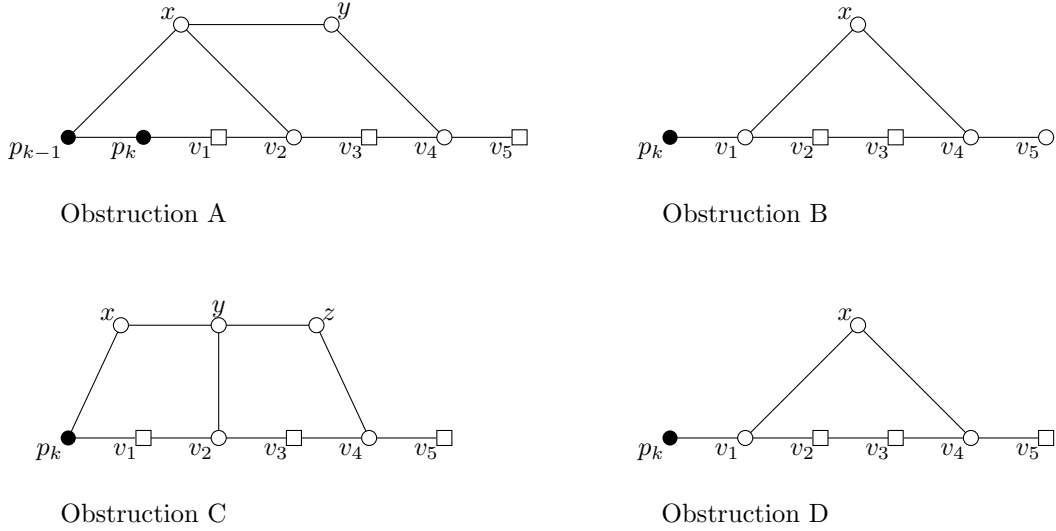}}
\caption{The obstructions.}
\label{fig-obs}
\end{figure}

By the absence of Obstruction D and Lemma~\ref{lemma-1chord}, $L'$ assigns a list of at least
two colors to all vertices of $V(G')\setminus V(P')$.  Since the girth of $G$ is at least $5$
and $|X|\le 4$, the induced subgraph $G[N]$ contains at most one edge.
By Lemma~\ref{lemma-1chord}, if $p_k\in V(G')$, then $p_k$ is not adjacent to any vertex $v$
with $|L(v)|=2$.  By Lemma~\ref{lemma-chords}, no vertex of $N$ is adjacent to a vertex with list of size two in $G$ or to $p_1$ or $p_k$
(for $v_{m+1}$ in case that $|L(v_{m+1})|=3$, the choice of $X$ implies that $m\le 3$ and $v_{m+2}\not\in I(G,P,L)$).
Therefore, $G'$ does not contain a path $u_1u_2u_3$ with $|L'(u_1)|=|L'(u_2)|=|L'(u_3)|=2$.

Suppose now that two vertices $x,y\in N$ are adjacent and there exists a path $xyzuw\subseteq G'$ with
$|L(z)|=3$ and one of the following holds:
\begin{itemize}
\item $|L(u)|=|L(w)|=2$, or
\item $\ell(P)=4$, $u=p_1$ and $|L(w)|=2$, or
\item $\ell(P)=4$, $w=p_1$ and $|L(u)|=2$.
\end{itemize}
Note that $x,y\neq v_{m+1}$ by the absence of Obstruction B.  Let $v_i\in X$ be the neighbor of $x$ and $v_j\in X$ the neighbor of $y$.
If $z\in V(F)$, then let $Q=v_jyz$, otherwise let $Q=v_jyzu$.  Note that $Q$ is a $2$- or $3$-chord.
Let $G_1,G_2\neq Q$ be the subgraphs of $G$ such that $G=G_1\cup G_2$, $Q=G_1\cap G_2$ and $P\subseteq G_1$.
Using Lemma~\ref{lemma-chords}, we conclude that $G_2$ consists of a single $5$-face and $|L(v_j)|=3$.
It follows that $i<j$.  Furthermore, consider the vertices $v_{j+1}$ and $v_{j+2}$ following $v_j$ in the boundary
of $F$.  If $z\in V(F)$, then both $v_{j+1}$ and $v_{j+2}$ have degree two.  If $z\not\in V(F)$, then
$v_{j+1}$ has degree two and $u=v_{j+2}$.  It follows that $|L(v_{j+1})|=2$ and either $|L(v_{j+2})|=2$ or $v_{j+2}=p_1$,
and since $L$ is a valid list assignment, $u=v_{j+2}$ and $w=v_{j+1}$.  Observe that $w\neq p_1$.
The cycle $C=v_iv_{i+1}\ldots v_jyx$ has length at most
$6$, hence $C$ bounds a face by (\ref{cl-face}).
All vertices $v_t$ with $i<t<j$ have degree two, and thus $|L(v_t)|=2$.  As $G$ has girth at least $5$, $i\le j-2$,
hence $|L(v_{j-1})|=2$.  Since $L$ is a valid list assignment and $|L(v_{j+1})|=2$ and $v_{j+2}\in D$, we have
$i=j-2$ and $|L(v_i)|=|L(v_j)|=3$.  Examination of the possible choices of $X$ shows that these conditions
may only be satisfied if $j=m$.  However, in that case $w=v_{m+1}$ has degree two in $G$, and
hence $w\in X'$, contradicting the assumption that $w\in V(G')$.

Suppose now that $G'$ contains a path $u_1u_2u_3u_4u_5$ with $|L'(u_1)|=|L'(u_2)|=|L'(u_4)|=|L'(u_5)|=2$ and $|L'(u_3)|=3$.
Since $L$ is a valid list assignment and $p_k$ has no neighbor with list of size two in $G'$,
we may assume that $u_1,u_2\in N$ and $u_4,u_5\not\in N$.  However, this contradicts the previous paragraph.
We conclude that $L'$ is a valid list assignment for $G'$ with respect to the path $P'$.

Let us now consider a path $p_{k-1}u_2u_3$ with $|L'(u_2)|=|L'(u_3)|=2$.  Note that $u_2,u_3\neq p_k$.
By Lemmas~\ref{lemma-1chord} and \ref{lemma-chords}, we have $u_2,u_3\neq p_1$ and $|L(u_2)|=|L(u_3)|=3$,
and thus $u_2,u_3\in N$.  This is not possible, as $G$ does not contain 
Obstruction A.  Finally, consider a path $p_{k-1}u_2u_3u_4u_5$ with
$|L'(u_2)|=|L'(u_4)|=|L'(u_5)|=2$ and $|L'(u_3)|=3$. By Lemma~\ref{lemma-1chord}, we have either
$u_2=p_k$ or $u_2\in N$.  In the former case, Lemmas~\ref{lemma-1chord} and \ref{lemma-chords}
imply $u_4,u_5\in N$, which is a contradiction, since $G$ does not contain Obstruction C.
It follows that $u_2\in N$, and by Lemma~\ref{lemma-1chord}, $u_2\neq v_{m+1}$.
Let $v_i$ be the neighbor of $u_2$ in $X$.  The $2$-chord $p_{k-1}u_2v_i$
bounds a $5$-face by Lemma~\ref{lemma-chords}, hence $i=2$, $|L(v_1)|=2$ and $|L(v_2)|=3$.  Since
$|X|\le 4$ and $G$ has girth $5$, $u_4$ and $u_5$ cannot both belong to $N$.  Since no vertex of $N$
is adjacent to a vertex with list of size two not belonging to $N$, it follows that $u_4,u_5\not\in N$,
and thus $u_4=p_1$ or $|L(u_4)|=2$.  If $u_3\in V(F)$, then let $Q=p_{k-1}u_2u_3$, otherwise
let $Q=p_{k-1}u_2u_3u_4$.  Lemma~\ref{lemma-chords} applied to $Q$ implies that $Q$ together with a path
in the boundary of $F$ bounds a $5$-face.  However, this contradicts the existence of the edge $u_2v_2$.
We conclude that $p_{k-1}$ is not bad.

Let us summarize the results of the previous few paragraphs:
\begin{itemize}
\item If $G$ does not contain Obstruction D, then $L'$ assigns each vertex of $V(G')\setminus V(P')$ at least two colors.
\item If additionally $G$ does not contain Obstruction B, then $L'$ is a valid list assignment.
\item If additionally $G$ does not contain Obstructions A and C, then $p_{k-1}$ is not bad.
\end{itemize}

By the minimality of $G$, $\psi$ can be extended to an $L'$-coloring of $G'$.
This is a contradiction, and thus $G$ contains at least one of the obstructions.
Note that the obstructions are mutually exclusive, hence $G$ contains exactly one of them.
Furthermore, if $G$ does not contain Obstructions B and D (so that $L'$ is a valid list
assignment), then both $p_{k-1}$ and $p_{k-3}$ are bad.
Let us consider each obstruction separately:

\begin{description}
\item[Obstruction A.] Let us recall that this obstruction consists of {\em a path $p_{k-1}xy$ with $x,y\in N$}. 
By Lemma~\ref{lemma-chords}, $x$ is adjacent
to $v_2$, $|L(v_1)|=2$ and $|L(v_2)|=3$.  It follows that $m=4$ and $y$ is adjacent to $v_4$.
As $v_2v_3v_4yx$ is a $5$-face, $|L(v_3)|=2$.  As $p_k$ is not bad in $G$, $|L(v_4)|=3$.  Since $v_4\in X$,
$|L(v_5)|=2$.  Therefore, $N$ consists of $x$, $y$ and other neighbors of $v_4$.  Observe also that $L'$
is a valid list assignment.

Suppose that $p_1$ is bad in $G'$.
No vertex of $N$ is adjacent to $p_1$ or to a vertex $z$ with $|L(z)|=2$
by Lemma~\ref{lemma-chords} and $p_1$ is not bad in $G$.  Since $p_1$ is bad in $G'$,
there exists a path $p_1z_1z_2xy$ or a path $p_1z_1z_2yx$ with $|L(z_1)|=2$ and $|L(z_2)|=3$.
The former is not possible, as the $2$-chord $v_2xz_2$ (if $z_2\in V(F)$) or the $3$-chord $v_2xz_2z_1$ would bound a $5$-face
by Lemma~\ref{lemma-chords}, contradicting the existence of $y$.  In the latter case, if $z_2\in V(F)$, then
$z_2yv_4$ is a $2$-chord and by Lemma~\ref{lemma-chords}, $z_2yv_4v_5v_6$ is a $5$-face (where $v_6$ is the common neighbor
of $v_5$ and $z_2$ in $F$), $|L(v_6)|=2$, and the path $p_1z_1z_2v_6v_5$ shows that $p_1$ is bad in $G$.
Similarly, if $z_2\not\in V(F)$, then the $3$-chord $z_1z_2yv_4$ together with the path $v_4v_5z_1$ bounds a $5$-face
and the path $p_1z_1v_5$ shows that $p_1$ is bad in $G$.  This is a contradiction, hence $p_1$ is not bad
in $G'$.

If $k=4$, this implies that $\psi$ extends to an $L$-coloring of $G$.  Therefore, $k=5$.
Suppose now that a vertex $v\in N$ is adjacent to $p_2$.  The corresponding $2$-chord bounds
a $5$-face by Lemma~\ref{lemma-chords}, which excludes the case $v=x$.  If $v\neq y$ were a neighbor of $v_4$,
then $p_2p_3p_4xyv_4v$ would be a $7$-face by (\ref{cl-face}), implying that $y$ has degree two.
This is a contradiction, thus $v=y$.  In this case, Lemma~\ref{lemma-chords} implies that $v_5$ is adjacent to $p_1$,
and by (\ref{cl-face}) $G$ is the class A graph depicted in Figure~\ref{fig-classAB}
($L$ may differ from the list assignment shown in the figure).
Therefore, no vertex of $N$ is adjacent to $p_2$.

Suppose that $v\in N$ is adjacent to $p_3$.
As the girth of $G$ is at least $5$, $v\neq x,y$, thus $v$ is a neighbor of $v_4$ distinct from $y$. 
However, then $p_3p_4p_5v_1v_2v_3v_4v$ would be a separating $8$-cycle, which contradicts
(\ref{cl-face}).  Similarly, we conclude that the only neighbor of $p_4$ in $N$ is $x$.

Let $c_4\in L(v_4)\setminus L(v_5)$ and $c_3\in L(v_3)\setminus\{c_4\}$ be chosen arbitrarily.
Observe that we may assume that $\varphi_{\psi}(v_4)=c_4$ and $\varphi_{\psi}(v_3)=c_3$, independently on the precoloring $\psi$ of $P$.
Let $P_2=p_1p_2p_3p_4x$ and $L_2$ be the list assignment for the vertices of $V(G')\setminus V(P_2)$ obtained from $L$ by
removing $c_4$ from the lists of neighbors of $v_4$.  Let us remark that $L_2$ is $L'$ restricted to $V(G')\setminus V(P_2)$,
thus $L_2$ is a valid list assignment.  Furthermore, the list assignment $L_2$ is independent
on the choice of $\psi$.

We have shown that $p_1$ is not bad and $p_2$, $p_3$ and $p_4$ are
not adjacent to a vertex with list of size two, thus they are not bad, either.  Finally, $x$ has list
of size two in the valid list assignment $L'$, thus $x$ is not bad.  We conclude that a $P_2$-skeleton $G_2$ of $G'$
satisfies assumptions of Theorem~\ref{thm-main}.  Since $x$ is not adjacent to $p_1$, $G_2$ is class A or B.

If $G_2$ is class A, then let $J=\{p_1,p_2,p_4\}$.  Note that $p_1$ is adjacent to a vertex $z$ in $G_2$ such that $|L_2(z)|=2$
and $z_2\not\in N$, thus $|L(z)|=2$, and $p_5$ is adjacent to $v_1$ in $G$, which has $|L(v_1)|=2$.
If $G_2$ is class B, then let $J=\{p_1,p_3\}$.  Given a precoloring $\psi'$ of $P$ that is $J$-different from
$\psi^{(G',P_2,L_2)}$, we color $X$ according to $\varphi'=\varphi_{\psi'}$ and choose a color for $x$ from
$L(x)\setminus\{\psi'(p_4),\varphi'(v_2)\}$.  This precoloring of $P_2$ extends to an $L_2$-coloring of $G_2$, giving an $L$-coloring of $G$.

Choose $c_2\in L(v_2)\setminus L(v_3)$ and $c_1\in L(v_1)\setminus\{c_2\}$.
Suppose now that $\psi'$ is not necessarily $J$-different from $\psi^{(G',P_2,L_2)}$, but that $\psi'(p_5)\neq c_1$.
Consider the graph $G_3=G-\{p_5,v_1,v_2\}$ with list assignment $L''$ obtained
from $L$ by removing $c_2$ from the list of $x$.  Observe that this list assignment is valid and that
no vertex of $P_3=p_1p_2p_3p_4$ is bad, thus $\psi$ extends to an $L_3$-coloring of $G_3$.
We extend this coloring to $G$ by coloring $v_1$ by $c_1$ and $v_2$ by $c_2$.
It follows that $G$ is class A or B, with $\psi^{(G,P,L)}$ matching $\psi^{(G',P_2,L_2)}$
on $p_1$, $p_2$ and $p_3$, $\psi^{(G,P,L)}(p_5)=c_1$ and $\psi^{(G,P,L)}(p_4)$ chosen so that $\psi^{(G',P_2,L_2)}(p_4)\in \{\psi^{(G,P,L)}(p_4),\psi^{(G,P,L)}(p_5)\}$.

\item[Obstruction B.] That is, $G$ contains {\em a path $vxv_{m+1}$, where $v\in X$ and $x\in N$}.
By Lemma~\ref{lemma-chords}, $|L(v)|=|L(v_{m+1})|=3$,
$v=v_{m-2}$ and $|L(v_{m-1})|=|L(v_m)|=2$.  Since $v_{m+1}\not\in X$, the inspection of the
choice of $X$ shows that $m=3$ and $|L(v_{m+2})|=3$.

By the criticality of $G$, we have $L(v_2)\cap L(v_3)\neq \emptyset$.  If $L(v_2)=L(v_3)$, then $S_2=S_3=L(v_2)$,
otherwise let $S_2=L(v_2)\setminus L(v_3)$ and $S_3=L(v_3)\setminus L(v_2)$.
Let $G_2=G-\{v_2,v_3\}$, with the list assignment $L_2$ such that $L_2(u)=L(u)$ for $u\not\in \{v_1, v_4\}$
and $L_2(v_1)\subseteq L(v_1)$ and $L_2(v_4)\subseteq L(v_4)$ be lists of size two chosen as follows:
\begin{itemize}
\item If $|S_2\cap L(v_1)|\le 1$, then choose $L_2(v_1)$ disjoint from $S_1$, and choose $L_2(v_4)$ arbitrarily.
\item If $|S_3\cap L(v_4)|\le 1$, then choose $L_2(v_4)$ disjoint from $S_2$, and choose $L_1(v_4)$ arbitrarily.
\item Otherwise, $L(v_2)=L(v_3)=\{a,b\}\subseteq L(v_1)\cap L(v_4)$.  Set $L_2(v_1)=L(v_1)\setminus \{b\}$
and $L_2(v_4)=L(v_4)\setminus \{a\}$.
\end{itemize}
Observe that any $L_2$-coloring of $v_1$ and $v_4$ extends to an $L$-coloring of $v_2$ and $v_3$,
thus any precoloring of $P$ that extends to an $L_2$-coloring of $G_2$ also extends to an $L$-coloring of $G$.
By Lemma~\ref{lemma-1chord}, $L_2$ is a valid list assignment, and no vertex of $P$ other than $p_1$ or $p_k$ is
bad.  If $p_1$ or $p_k$ were bad, then there would exist a path $v_1uwy\subseteq G_2$ with
$|L(u)|=3$, $|L(w)|=2$ and either $y=p_1$ or $|L(y)|=2$.  However, the $2$-chord $v_1uw$ would contradict
Lemma~\ref{lemma-chords}.

By the minimality of $G$, we can apply Theorem~\ref{thm-main} to a $P$-skeleton of $G_2$.  Since $\psi$
does not extend to an $L_2$-coloring of $G_2$, we conclude that $\ell(P)=4$ and $G_2$ is class A or B.
If $G_2$ is class B, then $G$ is class B as well.
If $G$ is class A, then $p_1$ is adjacent to a vertex $v$ such that $|L_2(v)|=2$.  By Lemma~\ref{lemma-1chord},
$v\not\in\{v_1,v_4\}$, and thus $|L(v)|=2$.  Furthermore, there exists a coloring
$\psi_R=\psi^{(G_2,P,L_2)}$ of $P$ and a set $R=\{p_1,p_2,p_4,p_5\}$ such that any precoloring $\psi'$ of $P$ 
that is $R$-different from $\psi_R$ extends to an $L$-coloring of $G$.  Let us remark that $G$ is not class A,
since $p_5$ is not adjacent to a vertex with list of size two.  We postpone the discussion of this case for later.

\item[Obstruction C.] Let us recall that Obstruction C consists of {\em a path $xyz$ with $x$ adjacent to $p_k$, $y$ to $v_2$ and $z$ to $v_4$,
and $v_4\in X$}.  As $G$ does not contain separating $5$-cycles, $|L(v_1)|=|L(v_3)|=2$.  Since $v_4\in X$, we have $m=4$,
and the inspection of the choice of $X$ shows that $|L(v_2)|=|L(v_4)|=3$ and $|L(v_5)|=2$.
There is no edge other than $p_kx$ between $\{x, y,z\}$ and $V(P)$---the only cases that are
not excluded by Lemma~\ref{lemma-chords} and the assumption that the girth of $G$ is at least $5$
are the following:
\begin{itemize}
\item $y$ adjacent to $p_3$, but then $G$ would contain a separating cycle $p_3\ldots p_kv_1v_2y$ of length
at most $6$, contrary to (\ref{cl-face}).
\item $z$ adjacent to $p_i\in\{p_2, p_3,\ldots, p_k\}$, but then $p_ip_{i+1}\ldots p_kxyz$ would bound a face
of length at most $7$ by (\ref{cl-face}), implying that $x$ has degree two, which is a contradiction.
\end{itemize}
By Lemma~\ref{lemma-nott}, $L(v_1)=\{\psi(p_k), c_1\}$ for some color $c_1$.  Suppose first that
$L(v_2)\neq L(v_3)\cup\{c_1\}$.  We choose a color $c_2\in L(v_2)\setminus (L(v_3)\cup\{c_1\})$. 
Let $G_2=G-\{v_1,v_2,v_3\}$ and $L_2$ be the list assignment such that $L_2(y)=L(y)\setminus \{c_2\}$,
$L_2(p_k)=\{\psi(p_{k-1}),\psi(p_k)\}$, $L_2(p_1)=\{\psi(p_1),\psi(p_2)\}$ and $L_2(v)=L(v)$ for any other vertex $v$.
Observe that $L_2$ is a valid list assignment for $G_2$ with respect to the path $p_2\ldots p_{k-1}$
and $p_{k-1}$ is not bad.  By the minimality of $G$, $\psi$ extends to an $L_2$-coloring of $G_2$, giving an $L$-coloring of $G$,
which is a contradiction.  Therefore, $L(v_2)=L(v_3)\cup\{c_1\}$, and thus $\{\psi(p_k)\}=L(v_1)\setminus (L(v_2)\setminus L(v_3))$.
This implies that any precoloring $\psi'$ of $P$ with $\psi'(p_k)\neq \psi(p_k)$ extends to an $L$-coloring of $G$.

Let $G_3=G-\{v_1,v_2,v_3,y\}$ and choose $c\in L(y)\setminus L(v_3)$.
Let $N_3$ be the set of neighbors of $y$ in $G$, excluding $v_2$.
Let $L_3$ be the list assignment for $V(G_3)\setminus V(P)$ obtained from $L$ by removing $c$ from the lists of vertices in $N_3$.
The vertices of $N_3$ form an independent set.  As we observed before, $x$ is not adjacent to a vertex
of $P$ other than $p_k$ and $z$ is not adjacent to any vertex of $P$.  A vertex $v\in N_3\setminus\{x,z\}$ is not adjacent to $p_1$
by Lemma~\ref{lemma-chords}, and $v$ is not adjacent to $p_i$ with $i\ge 2$, since otherwise the open disk bounded by
$p_i\ldots p_kv_1v_2yv$ would contain the vertex $x$, contrary to (\ref{cl-face}).
By Lemma~\ref{lemma-chords}, no vertex of $N_3$ is adjacent to a vertex $w$ with $|L(w)|=2$
and there does not exist a path $xw_1w_2$ with $|L(w_2)|=2$.
It follows that $L_3$ is a valid list assignment for $G_3$ and no vertex of $P$ is bad.
Note that $\psi$ does not extend to an $L_3$-coloring of $G_3$, hence
by the minimality of $G$, we conclude that $\ell(P)=4$ and $G_3$ is class A or B,
with $\psi_3=\psi^{(G_3,P,L_3)}$.  If $G_3$ is class A, then let $J=\{p_1,p_2,p_4\}$;
in this case, $p_1$ has a neighbor $v$ with $|L_3(v)|=2$, and since no vertex of $N_3$ is
adjacent to $p_1$, $|L(v)|=2$.  Furthermore, $p_5$ is adjacent to the vertex $v_1$ with
list of size two.  If $G_3$ is class B, then let $J=\{p_1,p_3\}$.

Consider a precoloring $\psi'$ of $P$ with $\psi'(p_5)=\psi(p_5)$, such that $\psi'$ is $J$-different from $\psi_3$.
This precoloring extends to an $L_3$-coloring of $G_3$.  Furthermore, it also extends
to an $L$-coloring of $G$:  We color $y$ by $c$, $v_1$ by $c_1$ and $v_3$ by a color $c_3\in L(v_3)$ different from the color
of $v_4$. Observe that $c=c_1$ or $c\not\in L(v_2)$, thus we can color $v_2$ by a color in $L(v_2)\setminus\{c_1,c_3\}$.
Let $\psi^{(G,P,L)}$ match $\psi_3$ on $p_1$, $p_2$ and $p_3$, $\psi^{(G,P,L)}(p_5)=\psi(p_5)$ and choose
$\psi^{(G,P,L)}(p_4)$ so that $\psi_3(p_4)\in\{\psi^{(G,P,L)}(p_4),\psi^{(G,P,L)}(p_5)\}$.
We conclude that if $G_3$ is class A, then $G$ is class A, and if $G_3$ is class B, then $G$ is class B.

\item[Obstruction D.] I.e., {\em a vertex $x\in N$ has two neighbors $v_i,v_j\in X$}.  Assume that $i<j$.
As $G$ has girth at least $5$ and $|X|\le 4$, $i=1$ and $j=4$, $|L(v_1)|=|L(v_4)|=3$ and $|L(v_2)|=|L(v_3)|=2$. 
The inspection of the choice of $X$ implies that $|L(v_5)|=2$.  By Lemma~\ref{lemma-chords}, $x$ is not adjacent to $p_1$ and $p_2$.
Since the girth of $G$ is at least $5$, $x$ is not adjacent to $p_k$ and $p_{k-1}$.

Suppose that $x$ is not adjacent to $p_3$.  Let us recall that $\varphi$ is a coloring of $X$, $G'=G-X'$ and
$L'$ is the list assignment obtained from $L$ by removing the colors of
vertices of $X$ from the lists of their neighbors and altering the list of $p_k$.  Choose a color $c\in L(x)\setminus\{\varphi(v_1),\varphi(v_4)\}$.
Let $G_2=G'-x$ and let $L_2$ be the assignment obtained from $L'$ by removing $c$ from the lists of neighbors of $x$.
Let $N_2$ be the set of vertices of $V(G_2)\setminus (V(P)\cup \{v_5\})$ that are adjacent to $v_1$, $x$ or $v_4$ in $G$.
Each vertex in $N_2$ is adjacent to only one of $v_i$, $x$ or $v_j$, as $G$ has girth at least $5$.
Furthermore, the vertices in $N_2$ form an independent set---if vertices $z_1,z_2\in N_2$ were adjacent, then, since
the girth of $G$ is at least $5$, say $z_1$ would be adjacent to $v_1$ and $z_2$ to $v_4$.  However, by (\ref{cl-face})
$v_1xv_4z_2z_1$ would bound a face and $x$ would have degree two.  Similarly, no vertex of $N_2$ is adjacent to $p_k$ or $v_5$.
By Lemma~\ref{lemma-chords}, for any $v\in N_2$ we have that $|L(v)|=3$ and that $v$ has no neigbor in $D$.
We conclude that $L_2$ is a valid list assignment to $G_2$ with respect to the path $P'$.

If $\psi$ extended to an $L_2$-coloring of $G_2$, then it would also extend to an $L$-coloring of $G$, hence this is not the case.
By the minimality of $G$, we can apply Theorem~\ref{thm-main} to a $P'$-skeleton of $G_2$, and we conclude that
$p_{k-1}$ is bad in $G_2$ with the list assignment $L_2$.
Since $N_2\cup \{p_k\}$ forms an independent set and no vertex of this set is adjacent to another vertex with list of size two,
it follows that there exists a path $p_{k-1}z_1z_2z_3\subseteq G_2$ with $z_1\in N_2\cup \{p_k\}$, $|L(z_2)|=3$ and either
$z_3=p_1$ or $|L(z_3)|=2$.  However, this contradicts Lemma~\ref{lemma-chords}.

Therefore, $x$ is adjacent to $p_3$.  Since $x$ is not adjacent to $p_{k-1}$, it follows that $\ell(P)=4$.
By (\ref{cl-face}), $p_3p_4p_5v_1x$ is a $5$-face.
Suppose that $L(v_2)\neq L(v_3)$ or $L(v_1)\neq L(v_2)\cup\{\psi(p_5)\}$.  Then there exists a color
$c_1\in L(v_1)\setminus\{\psi(p_5)\}$ such that for any $c_4\in L(v_4)$, the path $v_1v_2v_3v_4$ can be $L$-colored so that $v_1$ has color $c_1$
and $v_4$ has color $c_4$.  Let $G_3=G-\{p_4, p_5, v_1,v_2,v_3\}$ and let $L_3$ be the list assignment obtained
from $L$ by removing $c_1$ from the list of $x$, and setting $L_3(p_1)=\{\psi(p_1),\psi(p_2)\}$.
By Lemma~\ref{lemma-chords}, $x$ is not adjacent to $p_1$ or a vertex $v$ with $|L(v)|=2$, hence $L_3$ is a valid list assignment
for $G_3$ with respect to the path $p_2p_3$.  By the minimality of $G$, the precoloring of $p_2p_3$ given by $\psi$ extends to an
$L_3$-coloring of $G_3$, and further to an $L$-coloring of $G$, which is a contradiction.
Therefore, $L(v_1)=\{\psi(p_5),c_2,c_3\}$ and $L(v_2)=L(v_3)=\{c_2,c_3\}$ for some colors $c_2$ and $c_3$,
and $\{\psi(p_5)\}=L(v_1)\setminus L(v_2)$.  It follows that any precoloring $\psi'$ of $P$ that is $\{p_5\}$-different
from $\psi$ extends to an $L$-coloring of $G$.

Furthermore, $L(x)=\{\psi(p_3),c_2,c_3\}$, as 
if say $c_2\not\in L(x)$, then we could instead define $L_3(x)=L(x)\setminus \{c_3\}$,
and if $\psi(p_3)\not\in L(x)$, then we could define $L_3(x)=(L(x)\setminus \{c_2,c_3\})\cup \{\psi(p_3)\}$,
obtaining a contradiction in the same way.  Therefore, $\{\psi(p_3)\}=L(x)\setminus L(v_2)$, and
any precoloring $\psi'$ of $P$ that is $\{p_3\}$-different
from $\psi$ extends to an $L$-coloring of $G$.  Let $\psi_R=\psi$ and $R=\{p_3,p_5\}$.  
\end{description}

We proved that $\ell(P)=4$.  Furthermore, we proved that $G$ does not contain Obstructions A and C,
and if $G$ contains Obstruction B or D, there exists a set $R\subseteq V(P)$ and coloring $\psi_R$ of $P$
such that $p_5\in R$, $\{p_3,p_4\}\cap R\neq\emptyset$, any precoloring $\psi'$ of $P$ that is
$R$-different from $\psi_R$ extends to an $L$-coloring of $G$, and if $p_3\not\in R$, then $p_1$
is adjacent to a vertex with list of size two.

By symmetry of the path $P$, there exists a set $S\subseteq V(P)$ and coloring $\psi_S$ of $P$
such that $p_1\in S$, $\{p_2,p_3\}\cap S\neq\emptyset$, any precoloring $\psi'$ of $P$ that is
$S$-different from $\psi_S$ extends to an $L$-coloring of $G$, and if $p_3\not\in S$, then $p_5$
is adjacent to a vertex with list of size two.

If $p_3\in R$, then $G$ is class B, with $\psi^{(G,P,L)}$ matching $\psi_R$ on $p_3$ and $p_5$
and $\psi^{(G,P,L)}(p_1)=\psi_S(p_1)$.  Symmetrically, if $p_3\in S$, then $G$ is class B.
If $p_3\not\in R\cup S$, then $G$ is class A, with $\psi^{(G,P,L)}$ matching $\psi_R$ on $p_4$ and $p_5$
and $\psi_S$ on $p_1$ and $p_2$.  This is a contradiction.
\end{proof}

\section{Critical graphs with outer face of length $12$}\label{sec-cr12}

Theorem~\ref{thm-cyclesstr} (and Lemma~\ref{lemma-cycles}) provides a characterization
of the plane $F$-critical graphs of girth $5$, where $F$ is the
outer face of length at most $12$.  If $\ell(F)\le 11$, the complete
list of $F$-critical graphs is provided, however for $\ell(F)=12$,
only a necessary condition (every second vertex of $F$ is a $2$-vertex
incident with a $5$-face) is given for one subclass of the critical graphs.
Here, we show that this subclass in fact contains only one graph.

\begin{lemma}\label{lemma-cyc12}
Let $G$ be a plane graph of girth at least $5$, with the outer face $F$ bounded by an induced cycle
of length most $12$.  Furthermore, suppose that every second vertex of $F$ has degree two and is incident with a $5$-face.
Let $L$ be a list assignment of lists of size three to the vertices of $V(G)\setminus V(F)$.  
If $G$ is proper $F$-critical, then $G$ is isomorphic to the graph in Figure~\ref{fig-crit12}.
\end{lemma}
\begin{proof}
Let $G$ be a graph satisfying the assumptions of the lemma, and assume as the induction hypothesis
that any such graph $G'$ with $|V(G')|<|V(G)|$ is isomorphic to the graph in Figure~\ref{fig-crit12}.
Let $F=v_1v_2\ldots v_{12}$, where $v_2$, $v_4$, \ldots, $v_{12}$ are vertices of degree two incident with $5$-faces.
In particular, $v_1$, $v_3$, \ldots, $v_{11}$ have degree at least three.
\claim{cl-2ch}{The face $F$ has no $2$-chord.} Otherwise, we may assume that there exists a vertex $v$
adjacent to $v_1$ and $v_k$ for $5\le k\le 7$.  We may also assume that $v$ is not adjacent
to a vertex $v_i$ with $2\le i\le k-1$, thus $C=v_1v_2\ldots v_kv$ is an induced cycle of length at most $8$.
Since $v_2$ is incident with a $5$-face, the open disk bounded by $C$ contains at least one vertex, contradicting (\ref{cl-face}).

Suppose that there exists a $3$-chord $Q=v_ixyv_j$ such that $|i-j|\neq 2$, i.e., such that no cycle of $Q\cup F$ bounds a $5$-face.
We may assume that $i=1$ and $j\le 7$.  By (\ref{cl-2ch}), the cycle $C=v_1\ldots v_jyx$ is induced, and since $v_2$ is incident with a $5$-face,
the open disk bounded by $C$ contains at least one vertex.
By Lemmas~\ref{lemma-crs} and \ref{lemma-cycles}, $j=7$ and there is exactly one vertex $v$ of degree three in the open disk bounded $C$.
However, this is not possible, as $v$ cannot have two neighbors in $F$.
It follows that \claim{cl-3ch}{if $Q$ is a $3$-chord of $F$, then $Q\cup F$ contains a $5$-cycle.}

Consider now a $4$-chord $v_ixyzv_j$.  Again, we assume that $i=1$ and $j\le 7$.
Observe that the cycle $C=v_1\ldots v_jzyx$ is either the union of two $5$-faces (with $j=5$ and $y$ adjacent to $v_3$),
or induced.  Assume that $C$ is induced.  As in the proof of (\ref{cl-2ch}), we exclude the case $j\le 6$, thus $j=7$.
Let $C'=v_7v_8\ldots v_{12}v_1xyz$.  By (\ref{cl-2ch}) and (\ref{cl-3ch}), $C'$ is an induced cycle.
We apply Lemmas~\ref{lemma-crs} and \ref{lemma-cycles} to the $10$-cycles $C$ and $C'$.
By the constraints on the degrees of vertices and sizes of the faces incident with $F$,
we conclude that there are the following possibilities for $C$ (and symmetrically, for $C'$):

\begin{itemize}
\item[(a)] there is a $5$-cycle inside $C$, and the vertices of this $5$-cycle are adjacent to $v_1$, $v_3$, $v_5$, $v_7$ and $y$, or
\item[(b)] there are two adjacent vertices $u_1$ and $u_2$ inside $C$, $u_1$ is adjacent to $v_3$ and $x$ and $u_2$ is adjacent
to $v_5$ and $z$.
\end{itemize}

As each of $x$, $y$ and $z$ has degree at least $3$, the configuration (a) must appear in $C$ and the configuration (b) in $C'$
(or vice versa), implying that $G$ is the graph depicted in Figure~\ref{fig-crit12}.
Therefore, we can assume that \claim{cl-4ch}{any $4$-chord together with a path in $F$ bounds a cycle $K$ such that the closed disk
bounded by $K$ is a union of two $5$-faces.}

If all the vertices $v_1$, $v_3$, \ldots, $v_{11}$ had degree three, then $G-V(F)$ would be a $6$-cycle $K$
and all vertices of $K$ would have degree three in $G$, contradicting Lemma~\ref{lemma-gallai}.
Therefore, assume that $v_3$ has degree at least $4$.
Consider a coloring $\varphi$ of $F$ that does not extend to an $L$-coloring of $G$.  Let $G'=G-\{v_{12},v_1,v_2,v_3,v_4,v_5,v_6\}$
and let $L'$ be the list assignment obtained from $L$ by removing the colors of $v_1$, $v_3$ and $v_5$ from the lists of their neighbors
and setting $L'(v_7)=\{\varphi(v_7),\varphi(v_8)\}$ and $L'(v_{11})=\{\varphi(v_{11}),\varphi(v_{10})\}$.  As $\varphi$ does not
extend to an $L$-coloring of $G$, $G'$ together with $L'$ and the path $P=v_8v_9v_{10}$ must violate assumptions of
Theorem~\ref{thm-main}.  Observe that as $v_3$ has degree at least $4$, (\ref{cl-2ch}), (\ref{cl-3ch}) and (\ref{cl-4ch})
imply that $L'$ is a valid list assignment.  It follows that both $v_8$ and $v_{10}$ are bad.  Note that $v_7$
is the only vertex with list of size two adjacent to $v_8$, and by (\ref{cl-2ch}), $v_7$ is not adjacent to any vertex with list of size
two.  Therefore, there exists a path $v_8v_7xyz$ with $|L'(y)|=|L'(z)|=2$.  By (\ref{cl-2ch}) and (\ref{cl-3ch}), $y$ is adjacent to
$v_5$, $z$ is adjacent to $v_3$ and $v_5$ has degree three.  Symmetrically, since $v_{10}$ is bad, $v_1$ has degree three.
Similarly, \claim{cl-pn}{if $v_i\in V(F)$ has degree greater than three, then $v_{i-2}$ and $v_{i+2}$ have degree three.}

For every vertex $v_i\in V(F)$ of degree three, let $z_i$ be the neighbor of $v_i$ that is not incident with $F$.
Consider now the case that $v_7$, $v_9$ and $v_{11}$ have degree three.  Then, $G$ contains an $8$-cycle
$C=v_3z_3'z_5z_7z_9z_{11}z_1z_3''$, where $z_3'$ and $z_3''$ are neighbors of $v_3$.
By (\ref{cl-face}), at least one of $z_3'$ and $z_3''$ has degree two, which is a contradiction.

Suppose now that $v_9$ and $v_{11}$ have degree three, and thus $v_7$ has degree greater than $4$.
Consider the $10$-cycle $C=v_3z'_3z_5z'_7v_7z''_7z_9z_{11}z_1z''_3$, where $z'_7$ and $z''_7$ are neighbors of $v_7$.
Since $z'_3$, $z''_3$, $z'_7$ and $z''_7$ have degree at least three, Lemmas~\ref{lemma-crs} and \ref{lemma-cycles}
imply that the open disk bounded by $C$ contains a $5$-cycle $D$
with vertices adjacent to $z'_3$, $z''_3$, $z'_7$, $z''_7$ and $z_{11}$.  However, then the subgraph $G-(V(F)\cup \{z''_7, z_9, z_{11}, z_1, z''_3\})$
contradicts Lemma~\ref{lemma-gallai}.

By symmetry, it is also not the case that both $v_7$ and $v_9$ have degree three.
Suppose that $v_9$ has degree greater than three.  By (\ref{cl-pn}), $v_7$ and $v_{11}$ have degree three.
We apply Lemmas~\ref{lemma-crs} and \ref{lemma-cycles} to the $10$-cycle $C=v_3z'_3z_5z_7z'_9v_9z''_9z_{11}z_1z''_3$,
where $z'_9$ and $z''_9$ are neighbors of $v_9$.
Since $z'_3$, $z''_3$, $z'_9$ and $z''_9$ have degree at least three, the open disk bounded by $C$ contains
two adjacent vertices $p_1$ and $p_2$, with $p_1$ adjacent to $z'_3$ and $z'_9$ and $p_2$ adjacent to $z''_3$ and $z''_9$.
However, the $4$-chord $v_3z'_3p_1z'_9v_9$ contradicts (\ref{cl-4ch}).

Therefore, we may assume that $v_7$ and $v_{11}$ have degree greater than three and $v_9$ has degree three.
Consider the induced $12$-cycle $C=v_3z'_3z_5z'_7v_7z''_7z_9z'_{11}v_{11}z''_{11}z_1z''_3$, where
$z'_7$ and $z''_7$ are neighbors of $v_7$ and $z'_{11}$ and $z''_{11}$ are neighbors of $v_{11}$. Let $Z=\{z'_3,z''_3,z'_7,z''_7,z'_{11},z''_{11}\}$,
$Y=V(G)\setminus (V(F)\cup V(C))$ and $G_2=G-(V(F)\setminus V(C))$.  By (\ref{cl-2ch}) and (\ref{cl-3ch}), $C$ is an induced cycle.
By Lemma~\ref{lemma-cycles} and the induction hypothesis applied to $G_2$ whose outer face is bounded by $C$,
\begin{itemize}
\item[(a)] $G[Y]$ is a tree with at most $4$ vertices, or
\item[(b)] $G[Y]$ is a connected unicyclic graph consisting of a $5$-cycle $K$ and at most two other vertices, or
\item[(c)] $G_2$ is isomorphic to the graph in Figure~\ref{fig-crit12}.
\end{itemize}
By (\ref{cl-4ch}), each vertex in $Y$ is adjacent to at most one vertex of $C$.
On the other hand, each vertex of $Z$ has at least one neighbor in $Y$, hence $|Y|\ge |Z|=6$,
excluding the case (a).

Consider the case (b).  Since $|Y|\ge 6$, $G[Y]$ contains at least one vertex not belonging to $K$.
As $G[Y]$ is unicyclic, it contains a vertex $v$ of degree one.  As the degree of $v$ in $G$ is at least three, $v$ has
at least two neighbors in $C$, which is a contradiction.

Finally, consider the case (c), i.e., $G_2$ is the graph in Figure~\ref{fig-crit12}, with $Z$ corresponding to the
$\ge\!3$-vertices in the outer face.  We may assume that $z'_3$ and $z'_{11}$ are the vertices of degree four.
Let $H$ be the $2$-connected component of $G-(V(F)\cup \{z'_3,z'_{11}\})$ that contains $z''_3$.
Then all vertices of $V(H)$ have degree three in $G$ and $H$ is not an odd cycle, contradicting Lemma~\ref{lemma-gallai}.
\end{proof}

The description of the critical graphs with outer face of length at most $12$ follows:

\begin{corollary}\label{cor-desc12}
Let $G$ be a plane graph of girth at least $5$, with the outer face $F$ bounded by an induced cycle
of length at most $12$.  Let $L$ be a list assignment of lists of size three to the vertices of
$V(G)\setminus V(F)$.  If $G$ is proper $F$-critical graph, then
\begin{itemize}
\item $\ell(F)\ge 9$ and $G-V(F)$ is a tree with at most $\ell(F)-8$ vertices, or
\item $\ell(F)\ge 10$ and $G-V(F)$ is a connected graph with at most $\ell(F)-5$ vertices
containing exactly one cycle, and the length of this cycle is $5$, or
\item $G$ is the graph in Figure~\ref{fig-crit12}.
\end{itemize}
\end{corollary}

\section{4-critical graphs}\label{sec-4crit}

In this section, we prove Theorem~\ref{thm-numvert}.
Let us note that several of the ideas (the face weights, dealing with the possibly non-critical
graphs created by the reductions) used in this proof are inspired by the approach of
Dvo\v{r}\'ak et al.~\cite{proof-lincrit}.  However, the basic ideas of the proofs are quite different
(discharging vs. precoloring extension).  It should be noted that our approach gives better
bounds on the sizes of the critical graphs.

Let $w:Z^+\to R$ be the function defined in the following way: $w(x)=0$ for $x\le 4$,
$w(5)=1/7$ and $w(x)=x-5$ for $x\ge 6$.
Note the following basic properties of the function $w$:
\begin{itemize}
\item $w$ is non-decreasing
\item for every $x\ge 5$, $w(x)\le x-5 + w(5)$
\item for every $x<y$, $w(x)-w(x-1)\le w(y)-w(y-1)$
\end{itemize}
The consequence of the last of these properties is the following:
\claim{cl-msum}{If $x+y=z$ and $x,y\ge m$, then $w(x)+w(y)\le w(z-m)+w(m)$.}

Let $G$ be a plane graph with the outer face $F$,
and let $\FF(G)$ be the set of faces of $G$ excluding the outer face $F$.
Let the {\em weight} $w(G)$ of $G$ be defined as $w(G)=\sum_{f\in \FF(G)} w(\ell(f))$.

Let $E_1$ be the set of all cycles of length at least $5$.  Let $E_2$ be the set of
plane graphs $G$ of girth at least $5$ with outer face $F$ bounded by a cycle
such that $G-F$ consists of a chord of $F$.  Let $E_3$ be the set of
plane graphs $G$ of girth at least $5$ with outer face $F$ bounded by an induced cycle
such that $V(G)\setminus V(F)$ consists of a single vertex of degree three.
A graph $G$ is {\em exceptional} if $G\in E_1\cup E_2\cup E_3$.
Note that if $G$ is exceptional, then
\begin{itemize}
\item $w(G)\le w(\ell(F))$, and
\item if $G\not\in E_1$, then $w(G)\le w(\ell(F)-3)+w(5)$, and
\item if $G\not\in E_1\cup E_2$, then $w(G)\le w(\ell(F)-4)+2w(5)$.
\end{itemize}

We prove the following claim, which implies Theorem~\ref{thm-numvert}.

\begin{theorem}\label{thm-fincrit}
Let $G$ be a plane graph of girth at least $5$ with the outer face $F$, and $L$ an assignment of lists
of size three to vertices of $V(G)\setminus V(F)$.  If $G$ is a non-exceptional $F$-critical graph, then
$\ell(F)\ge 10$ and $w(G)\le w(\ell(F)-5)+5w(5)$.
\end{theorem}

Note that the bound in Theorem~\ref{thm-fincrit} is tight for the graph $G$ whose outer face $F$ is bounded
by an induced cycle, $G-V(F)$ is $5$-cycle $C$, every vertex of $C$ has degree three in $G$ and
$G$ has only one face of length greater than $5$ distinct from $F$.

Before we proceed with the proof of the theorem, let us introduce several definitions and auxiliary results.
Let $G$ be a plane graph of girth at least $5$ with the outer face $F$.
A {\em jump} in $G$ is a subgraph of $G$ consisting of two $5$-faces $v_1v_2v_3yx$
and $v_3v_4v_5zy$ such that the path $v_1v_2v_3v_4v_5$ (the {\em base} of the jump)
is a part of the facial walk of $F$ and $x,y,z\not\in V(F)$. The path $v_1xyzv_5$
is called the {\em body} of the jump.  The {\em internal vertices} of the jump
are $v_2$, $v_3$, $v_4$, $x$, $y$ and $z$.  Two jumps are {\em disjoint} if the sets of
their internal vertices are disjoint.  A {\em peeling} of $G$ is a subgraph $H\subseteq G$
obtained from $G$ by removing the internal vertices of the bases of at most two disjoint jumps.
Let $B$ be the outer face of $H$.  Note that $\ell(B)=\ell(F)$ and $w(G)\le w(H)+4w(5)$.
Also, by Lemma~\ref{lemma-crs}, if $G$ is $F$-critical, then $H$ is $B$-critical.
Let us now show that an $F$-critical graph $G$ contains one of several configurations, which
enable us to reduce it and apply induction.  The configurations discussed in the cases (g) and (h)
are illustrated in Figure~\ref{fig-last}.

\begin{lemma}\label{lemma-struct}
Let $G$ be a plane graph of girth at least $5$ with the outer face $F$ and $L$ an assignment of lists
of size three to vertices of $V(G)\setminus V(F)$.  If $G$ is a proper $F$-critical graph,
then at least one of the following holds:

\begin{itemize}
\item[(a)] the outer face of a peeling of $G$ is not bounded by an induced cycle, or
\item[(b)] the outer face of a peeling of $G$ has a $2$-chord, or
\item[(c)] $F$ contains two adjacent vertices of degree two, or
\item[(d)] a peeling $H$ of $G$ with the outer face $B$ has a $3$-chord $Q$ such that
no cycle in $B\cup Q$ distinct from $B$ bounds a face of $G$, or
\item[(e)] a peeling $H$ of $G$ with the outer face $B$ has a $4$-chord $Q=v_0v_1v_2v_3v_4$
such that for both cycles $K\subseteq B\cup Q$ distinct from $B$, the subgraph of $G$
drawn in the closed disk bounded by $K$ is equal neither to $K$ nor to $K$ with exactly
one chord incident with $v_2$, or
\item[(f)] there exists a path $uvwx\subseteq G$ such that $u,v,w,x\not\in V(F)$ and each of $u$, $v$, $w$ and $x$
has a neighbor in $F$, or
\item[(g)] there exists a $4$-chord $Q=v_0v_1v_2v_3v_4$ of the outer face $B$ of a peeling $H$ of $G$ such that a cycle
$C\subseteq B\cup Q$ distinct from $B$ bounds a face of $G$ and $H$ contains a jump intersecting $C$ in $v_0v_1$, or
\item[(h)] there exist a $4$-chord $Q=v_0v_1v_2v_3v_4$ of the outer face $B$ of a peeling $H$ of $G$ and $5$-faces
$C_1$ and $C_2$ of $H$ such that a cycle $C\subseteq B\cup Q$ distinct from $B$ bounds a face of $G$,
$|V(C_1\cap B)|=|V(C_2\cap B)|=3$, $C_1\cap C=v_0v_1$ and $C_2\cap C=v_3v_4$.
\end{itemize}
\end{lemma}
\begin{proof}
Suppose that none of (a-h) holds.  Then the outer face of $G$ is an induced cycle, and since $G$ is $F$-critical, $G$ is $2$-connected.
By Lemma~\ref{lemma-cycles}, since (b) is false, $\ell(F)\ge 10$.

Let $X$ be the set of vertices of $V(G)\setminus V(F)$ that have a neighbor in $F$.
Since (b) is false, each vertex of $X$ has exactly one neighbor in $F$.
Since (c) and (d) are false, if $x_1,x_2\in X$ are adjacent, then there exists a unique $5$-face
$f(x_1x_2)=x_1x_2v_1v_2v_3$, where $v_1v_2v_3$ is a part of the facial walk of $F$.
Similarly, if $x_1x_2x_3$ is a path with $x_1,x_2,x_3\in X$, then the $5$-faces $f(x_1x_2)$ and $f(x_2x_3)$ intersect $F$ in consecutive segments.
It follows that $G[X]$ is either a cycle, or a union of paths.
Note that $G[X]$ cannot be a cycle of length three, since the girth of $G$ is at least $5$.
Since (f) is false, $G[X]$ is a union of paths of length at most two.

For $i\in\{0,1,2\}$, let $X_i$ be the set of vertices $x\in X$ such that the maximal path
in $G[X]$ that contains $x$ has length $i$.  Note that each path $P$ in $G[X_2]$ corresponds
to a jump.  Let $Y$ be the set of vertices of $V(G)\setminus(V(F)\cup X_2)$
that have a neighbor in $X_2$.  Note that $X\cap Y=\emptyset$.

Suppose that a vertex $y\in Y$ has two neighbors in $X$.  Let $x_1$ be a neighbor of $y$ in $X_2$ and
$x_2\in X$ another neighbor in $X$.  Note that $F$ has a $4$-chord $Q=v_1x_1yx_2v_2$.  As $X\cap Y=\emptyset$, $y$ is not adjacent to a vertex of $F$,
and since (e) is false, $Q\cup F$ contains a cycle $K\neq F$ bounding a face.
However, the $4$-chord $Q$ together with the jump containing $x_1$ implies that $G$ satisfies (g), which is a contradiction.
Therefore, each vertex in $Y$ has exactly one neighbor in $X$ (and this neighbor belongs to $X_2$).

Consider now two adjacent vertices $y_1,y_2\in Y$.  For $i\in\{1,2\}$, let $J_i$ be the jump
in that $y_i$ has a neighbor.  Suppose that $J_1\neq J_2$, and let $H$ be the peeling of $G$
obtained by removing the internal vertices of the base of $J_1$.  Let $B$ be the outer face of $H$.
Then $B$ has a $4$-chord $Q=x_1y_1y_2x_2v$, with $x_i$ belonging to the body of $J_i$, for $i\in \{1,2\}$.
Note that $y_2$ does not have a neighbor in $B$, as $y_2\not\in X$, and since (e) is false,
$Q\cup B$ contains a cycle $K\neq B$ bounding a face of $G$.
However, $Q$ together with the jump $J_2$ implies that $G$ satisfies (g), which is a contradiction.
It follows that $J_1=J_2$, and as $G$ has girth at least $5$,
$y_1y_2$ together with a path in $X_2$ (a subpath of the body of $J_1$)
bounds a $5$-face $f(y_1y_2)$.  Since for any edge $y_1y_2\in E(G[Y])$, the neighbors of $y_1$ and $y_2$
in $X$ belong to the same jump, we conclude that $G[Y]$ does not contain a path of length two.

Let $Y_1\subseteq Y$ be the set of vertices that are incident with an edge in $G[Y]$.
Note that $G[X_1\cup Y_1]$ is $1$-regular.
Suppose that a vertex $v\in V(G)\setminus (V(F)\cup X\cup Y)$ has
two neighbors $z_1,z_2\in X_1\cup Y_1$.  For $i\in\{1,2\}$, let $z'_i$ be the neighbor
of $z_i$ in $X_1\cup Y_1$, and $v_i$ the neighbor of $z_i$ in $f(z_iz'_i)$ distinct
from $z'_i$.  There exists a peeling $H$ of $G$ with the outer face $B$ such that
$f(z_iz'_i)-\{z_i,z'_i\}$ is a subpath of $B$ for $i\in\{1,2\}$.  Then,
$Q=v_1z_1vz_2v_2$ is a $4$-chord of $B$.  As $v\not\in X\cup Y$, $v$ does not have a neighbor in $B$,
and since (e) is false, $Q\cup B$ contains a cycle $K\neq B$ bounding a face of $G$.
However, $Q$ together with the faces $f(z_1z'_1)$ and $f(z_2z'_2)$ implies that $G$ satisfies (h),
which is a contradiction.  Therefore, no vertex in $V(G)\setminus (V(F)\cup X\cup Y)$
has two neighbors in $X_1\cup Y_1$.

As $G$ is critical and $G\neq F$, there exists a precoloring $\psi$ of $F$ that does
not extend to an $L$-coloring of $G$.  Since each vertex of $X_2$ has list of size three,
only one neighbor in $F$ and $G[X_2]$ is a union of paths, there exists an $L$-coloring $\psi_1$
of $G[V(F)\cup X_2]$ extending $\psi$.  Let $G_1=G-(V(F)\cup X_2)$ and let $L_1$ be the list assignment for $G_1$
such that $L_1(v)=L(v)$ if $v$ has no neighbor in $V(F)\cup X_2$ and $L_1(v)=L(v)\setminus\{\psi_1(x)\}$ if
$x\in V(F)\cup X_2$ is the (unique) neighbor of $v$.  Note that each vertex of $G_1$ has list of size at least two,
and the set of vertices with lists of size two is a subset of $Z=X_0\cup X_1\cup Y$.
As we proved in the previous paragraphs, $G[Z]$ does not contain a path of length two
and there does not exist a path $z_1z_2vz_3z_4\subseteq G_1$ with
$z_1,z_2,z_4,z_5\in Z$.  Therefore, $G_1$ with the list assignment $L_1$ satisfies
assumptions of Theorem~\ref{thm-mainsim} and $G_1$ has an $L_1$-coloring $\varphi$.
However, $\psi_1\cup \varphi$ is an $L$-coloring of $G$ that extends $\psi$,
which is a contradiction.
\end{proof}

The following claims allow us to deal with the configurations described in Lemma~\ref{lemma-struct}.

\begin{lemma}\label{lemma-nontriv}
Let $G$ be a plane graph of girth at least $5$ with the outer face $F$ and $L$ an assignment of lists
of size three to vertices of $V(G)\setminus V(F)$.  Suppose that $G$ is a non-exceptional $F$-critical graph
and $H$ is a peeling of $G$.  Then $H$ is not exceptional.
\end{lemma}
\begin{proof}
Suppose that $H$ is exceptional.  Since $G$ is not exceptional, there exists a jump $J\subseteq G$ such that
the body $v_1xyzv_5$ of $J$ is a part of the boundary of the outer face of $H$.
As $H$ is exceptional, $x$ or $z$ (say $x$) has degree two in $H$.  However, then $x$ has degree two in $G$ as well,
contradicting the criticality of $G$.
\end{proof}

\begin{lemma}\label{lemma-a}
Let $G$ be a plane graph of girth at least $5$ with the outer face $F$ and $L$ an assignment of lists
of size three to vertices of $V(G)\setminus V(F)$.  Suppose that $G$ is a non-exceptional $F$-critical graph
and $H$ is a peeling of $G$ with the outer face $B$ such that $B$ is not an induced cycle.
Let $H_1$ and $H_2$ be induced subgraphs of $H$ such that $H=H_1\cup H_2$, $H_1\neq H_1\cap H_2\neq H_2$ and $H_1\cap H_2$ is
either a vertex of $B$, or a chord of $B$.  Let $B_i$ be the outer face of $H_i$ for $i\in \{1,2\}$.
If $H_1,H_2\in E_1$, then $\ell(F)=\ell(B_1)+\ell(B_2)$.
\end{lemma}
\begin{proof}
If $H_1\cap H_2$ is a chord of $B$, then $H\in E_2$ contrary to Lemma~\ref{lemma-nontriv}.
It follows that $\ell(F)=\ell(B_1)+\ell(B_2)$.
\end{proof}

\begin{lemma}\label{lemma-b}
Let $G$ be a plane graph of girth at least $5$ with the outer face $F$ and $L$ an assignment of lists
of size three to vertices of $V(G)\setminus V(F)$.  Suppose that $G$ is a non-exceptional $F$-critical graph
and $H$ is a peeling of $G$ with the outer face $B$ bounded by an induced cycle.
Let $Q$ be a $2$-chord of $B$ and $H_1,H_2\neq Q$ be induced subgraphs of $H$ such that $H=H_1\cup H_2$ and $H_1\cap H_2=Q$.
If $H_1\in E_1$, then $H_2\not\in E_1\cup E_2$.  If additionally $H_2\in E_3$, then $w(G)\le w(H)+2w(5)$.
\end{lemma}
\begin{proof}
Suppose that $H_1\in E_1$.
If $H_2\in E_1$, then the vertex $v\in V(Q)\setminus V(B)$ has degree two, contradicting the criticality of $G$.
If $H_2\in E_2$, then $H\in E_3$, contrary to Lemma~\ref{lemma-nontriv}.
Suppose for a contradiction that $H_2\in E_3$ and $w(G)>w(H)+2w(5)$, i.e., $H$ was obtained from $G$ by
removing the internal vertices of the bases of two jumps $J_1$ and $J_2$.  Since $v$ has degree at least
three, it follows that $V(H)\setminus V(B)$ consists of two adjacent vertices of degree three.
Let $x_i$ and $y_i$ be the internal vertices of the bodies of $J_i$ that have degree two in $J_i$, for $i\in \{1,2\}$.
Since $x_1$, $y_1$, $x_2$ and $y_2$ have degree greater than two in $G$, each of them is adjacent to
one of the vertices of $V(H)\setminus V(B)$.  However, then each vertex of $V(G)\setminus V(F)$ has degree three,
$G-V(F)$ is $2$-connected and not an odd cycle, which contradicts Lemma~\ref{lemma-gallai}.
\end{proof}

\begin{lemma}\label{lemma-c}
Let $G$ be a plane graph of girth at least $5$ with the outer face $F$ and $L$ an assignment of lists
of size three to vertices of $V(G)\setminus V(F)$.  Suppose that $G$ is a non-exceptional $F$-critical graph
that does not have properties (a) and (b) of Lemma~\ref{lemma-struct}, and that $v_1$ and $v_2$ are
two adjacent vertices of degree two in $G$.  Let $G_1$ be the graph obtained from $G$ by identifying $v_1$ with $v_2$,
and $F_1$ the outer face of $G_1$.  Then $\ell(F_1)=\ell(F)-1$, $G_1$ is a non-exceptional $F_1$-critical graph and the girth of $G_1$ is at least $5$.
\end{lemma}
\begin{proof}
Note that $v_1, v_2\in V(F)$, and thus $\ell(F_1)=\ell(F)-1$.  Let $v_0v_1v_2v_3$ be the subpath of $F$ containing $v_1$ and $v_2$.
Since $G$ does not satisfy (a) and (b), $v_0$ and $v_3$ do not have a common neighbor, and thus the girth of $G_1$ is at least $5$.
Also, for any precoloring $\psi$ of $F$ there exists a precoloring $\psi_1$ of $F_1$ matching $\psi$ on $V(F)\setminus \{v_1,v_2\}$,
and $\psi$ extends to an $L$-coloring of a subgraph of $G$ if and only if $\psi_1$ extends to an $L$-coloring of the corresponding
subgraph of $G_2$, thus $G_1$ is $F_1$-critical.  Since $G$ is not exceptional and subdividing an edge of the outer face of an
exceptional graph results in an exceptional graph, $G_1$ is not exceptional.
\end{proof}

\begin{lemma}\label{lemma-d}
Let $G$ be a plane graph of girth at least $5$ with the outer face $F$ and $L$ an assignment of lists
of size three to vertices of $V(G)\setminus V(F)$.  Suppose that $G$ is a non-exceptional $F$-critical graph
that does not have properties (a) and (b) of Lemma~\ref{lemma-struct} and let $H$ be a peeling of $G$ with the outer face $B$.
Let $Q$ be a $3$-chord of $B$ and $H_1,H_2\neq Q$ be the induced subgraphs of $H$ such that $H=H_1\cup H_2$ and $H_1\cap H_2=Q$.
If $H_1,H_2\not\in E_1$, then at least one of $H_1$ and $H_2$ is not exceptional.
\end{lemma}
\begin{proof}
Since $G$ does not have the properties (a) and (b), $H-V(B)$ is not a tree, thus $|V(H)\setminus V(B)|\ge 5$ and
$|V(H)\setminus V(B)\setminus V(Q)|\ge 3$.  This implies that at least one of $H_1$ and $H_2$ has at least two
vertices not incident with its outer face, and thus it is not exceptional.
\end{proof}

\begin{lemma}\label{lemma-e}
Let $G$ be a plane graph of girth at least $5$ with the outer face $F$ and $L$ an assignment of lists
of size three to vertices of $V(G)\setminus V(F)$.  Suppose that $G$ is a non-exceptional $F$-critical graph
that does not have properties (a) and (b) of Lemma~\ref{lemma-struct} and let $H$ be a peeling of $G$ with the outer face $B$.
Let $Q=v_0v_1v_2v_3v_4$ be a $4$-chord of $B$ and $H_1,H_2\neq Q$ be induced subgraphs of $H$ such that $H=H_1\cup H_2$ and $H_1\cap H_2=Q$.
Suppose that for $i\in\{1,2\}$, $H_i\not\in E_1$ and if $H_i\in E_2$, then $v_2$ has degree two in $H_i$.
Then $H_1,H_2\not\in E_2$, and at least one of $H_1$ and $H_2$ is not exceptional.
\end{lemma}
\begin{proof}
Suppose that say $H_1\in E_2$, and let $B_1$ be the outer face of $H_1$.
Since the chord of $B_1$ is not incident with $v_2$, $B$ has either a chord or a $2$-chord,
contradicting the assumption that $G$ does not have properties (a) and (b).

Since the girth of $G$ is at least $5$, if $H_1\in E_3$, then at least two of $v_1$, $v_2$ and $v_3$ have
degree two in $H_1$.  Symmetrically, if $H_2\in E_3$, then at least two of $v_1$, $v_2$ and $v_3$ have
degree two in $H_2$.  Therefore, if $H_1,H_2\in E_3$, then at least one of $v_1$, $v_2$ and $v_3$ has
degree two in $G$, which is a contradiction.  It follows that at most one of $H_1$ and $H_2$ is exceptional.
\end{proof}

We are now ready to prove the main theorem of this section.

\begin{proof}[Proof of Theorem~\ref{thm-fincrit}]
We proceed by induction on $\ell(F)$ and the number of edges of $G$.  If $\ell(F)\le 11$, then the claim follows from
Lemma~\ref{lemma-cycles}.  Suppose that $\ell(F)\ge 12$ and Theorem~\ref{thm-fincrit} holds for
all graphs with the outer face of length at most $\ell(F)-1$, as well all graphs with outer the face
of length $\ell(F)$ and fewer edges than $G$.

The graph $G$ satisfies at least one of the conclusions of Lemma~\ref{lemma-struct}.
If $G$ has the property (a), then let $H_1,H_2\subseteq G$ be the subgraphs of $G$ as in
Lemma~\ref{lemma-a}, with outer faces $B_1$ and $B_2$.  Note that $\ell(B_1)+\ell(B_2)\le \ell(F)+2$,
and thus $\ell(B_1),\ell(B_2)\le \ell(F)-3$.  By Lemma~\ref{lemma-crs}, $H_i$ is $B_i$-critical
for $i\in\{1,2\}$.  If $\{H_1,H_2\}\not\subseteq E_1$, then by symmetry assume that $H_1\not\in E_1$.
By the induction hypothesis, $w(H_1)\le w(\ell(B_1)-3)+w(5)$,
and thus $w(G)\le w(H)+4w(5)=w(H_1)+w(H_2)+4w(5)\le w(\ell(B_1)-3)+w(\ell(B_2))+5w(5)$.  Note that $\ell(B_1)\ge 8$ and
$\ell(B_2)\ge 5$, thus by (\ref{cl-msum}),
$w(\ell(B_1)-3)+w(\ell(B_2))\le w(\ell(B_1)+\ell(B_2)-8)+w(5)\le w(\ell(F)-6)+w(5)\le w(\ell(F)-5)$.
We conclude that $w(G)\le w(\ell(F)-5)+5w(5)$.

On the other hand, if $H_1, H_2\in E_1$, then using Lemma~\ref{lemma-a}, we obtain
$w(G)=w(H_1)+w(H_2)+4w(5)=w(\ell(B_1))+w(\ell(B_2))+4w(5)\le
w(\ell(B_1)+\ell(B_2)-5)+5w(5)=w(\ell(F)-5)+5w(5)$.
Therefore, we may assume that $G$ does not have the property (a), that is, any peeling of $G$
is bounded by an induced cycle.

Suppose that $G$ has the property (b).  Let $Q$, $H_1$ and $H_2$ be the subgraphs of $G$ as in
Lemma~\ref{lemma-b}, and let $B_1$ and $B_2$ be the outer faces of $H_1$ and $H_2$, respectively.
Note that $\ell(B_1)+\ell(B_2)=\ell(F)+4$, and since the girth of $G$ is at least $5$, it follows that
$\ell(B_1),\ell(B_2)<\ell(F)$.  If $\{H_1,H_2\}\cap E_1\neq \emptyset$, then by symmetry assume that
$H_1\in E_1$.  By Lemma~\ref{lemma-b}, $H_2\not\in E_1\cup E_2$.  By the induction hypothesis
and Lemma~\ref{lemma-b}, if $H_2\in E_3$, then $w(G)\le w(H_1)+w(H_2)+2w(5)\le w(\ell(B_1))+w(\ell(B_2)-4)+4w(5)$.
If $H_2\not\in E_3$, then $w(G)\le w(H_1)+w(H_2)+4w(5)\le w(\ell(B_1))+w(\ell(B_2)-5)+9w(5)\le w(\ell(B_1))+w(\ell(B_2)-4)+4w(5)$.
By (\ref{cl-msum}), $w(\ell(B_1))+w(\ell(B_2)-4)\le w(\ell(B_1)+\ell(B_2)-9)+w(5)=w(\ell(F)-5)+w(5)$,
and thus $w(G)\le w(\ell(F)-5)+5w(5)$.

On the other hand, it $H_1,H_2\not\in E_1$, then
$w(G)\le w(H_1)+w(H_2)+4w(5)\le w(\ell(B_1)-3)+w(\ell(B_2)-3)+6w(5)\le w(\ell(B_1)+\ell(B_2)-11)+7w(5)=w(\ell(F)-7)+7w(5)\le w(\ell(F)-5)+5w(5)$.
Therefore, we may assume that $G$ does not have the property (b).

Suppose that $v_1$ and $v_2$ are adjacent vertices of degree two in $G$, and let $G_1$ and $F_1$ be as in Lemma~\ref{lemma-c}.  Let $f\neq F$ be
the face of $G$ incident with $v_1v_2$.  Then $\ell(f)\ge 6$ and $w(G)=w(G_1)+w(\ell(f))-w(\ell(f)-1)$.  By the induction hypothesis,
$w(G_1)\le w(\ell(F_1)-5)+5w(5)=w(\ell(F)-6)+5w(5)$.  This implies that $\ell(f)-1 < \ell(F_1)=\ell(F)-1$.
We conclude that $w(G)\le w(\ell(F)-6)+w(\ell(f))-w(\ell(f)-1)+5w(5)\le w(\ell(F)-5)+5w(5)$.
Therefore, assume that $G$ does not have the property (c).

Suppose that $G$ has the property (d).  Let $Q$, $H_1$ and $H_2$ be the subgraphs of $G$ as in
Lemma~\ref{lemma-d}, and let $B_1$ and $B_2$ be the outer faces of $H_1$ and $H_2$, respectively.
Note that $\ell(B_1)+\ell(B_2)=\ell(F)+6$.  Since $H_1,H_2\not\in E_1$, we have $\ell(B_1),\ell(B_2)\ge 8$,
and thus $\ell(B_1),\ell(B_2)<\ell(F)$.  By Lemma~\ref{lemma-d}, we may assume that $H_2$ is not exceptional,
and thus $\ell(B_2)\ge 10$.  By the induction hypothesis,
$w(G)\le w(H_1)+w(H_2)+4w(5)\le w(\ell(B_1)-3)+w(\ell(B_2)-5)+10w(5)$.
By (\ref{cl-msum}), $w(\ell(B_1)-3)+w(\ell(B_2)-5)\le w(\ell(B_1)+\ell(B_2)-13)+w(5)$,
and thus $w(G)\le w(\ell(F)-7)+11w(5)\le w(\ell(F)-5)+5w(5)$.  Therefore, assume that $G$ does not have
the property (d).

Suppose that $G$ has the property (e).  Let $Q$, $H_1$ and $H_2$ be the subgraphs of $G$ as in
Lemma~\ref{lemma-e}, and let $B_1$ and $B_2$ be the outer faces of $H_1$ and $H_2$, respectively.
By Lemma~\ref{lemma-e} and symmetry, we assume that $H_1\not\in E_1\cup E_2$
and $H_2$ is not exceptional, and thus $\ell(B_1)\ge 9$ and $\ell(B_2)\ge 10$.
Note that $\ell(B_1)+\ell(B_2)=\ell(F)+8$, hence $\ell(B_1),\ell(B_2)<\ell(F)$.
By the induction hypothesis, $w(G)\le w(H_1)+w(H_2)+4w(5)\le w(\ell(B_1)-4)+w(\ell(B_2)-5)+11w(5)$.
By (\ref{cl-msum}), $w(\ell(B_1)-4)+w(\ell(B_2)-5)\le w(\ell(B_1)+\ell(B_2)-14)+w(5)$,
and thus $w(G)\le w(\ell(F)-6)+11w(5)\le w(\ell(F)-5)+5w(5)$.  It follows that we can assume that
$G$ does not have the property (e).

Suppose that $G$ has the property (f).  Since $G$ does not have the properties (a-d),
there exists a path $v_0v_1\ldots v_6\subseteq F$ such that $u$ is adjacent to $v_0$,
$v$ to $v_2$, $w$ to $v_4$ and $x$ to $v_6$, and the closed disk bounded by
$v_0v_1\ldots v_6xwvu$ consists of three $5$-faces of $G$.  Let $G'=G-\{v_1,v_2,\ldots, v_5\}$
and let $F'$ be the outer face of $G'$.  Observe that $G'$ is $F'$-critical, $\ell(F')=\ell(F)-1$
and $w(G)=w(G')+3w(5)$.  Since $u$ and $x$ have degree at least three in $G$, they have degree
at least three in $G'$.  Also, $u$ is not adjacent to $x$, since the girth of $G$ is at least $5$,
thus $G'\not\in E_1\cup E_2$.  By the induction hypothesis,
$w(G')\le w(\ell(F')-4)+2w(5)=w(\ell(F)-5)+2w(5)$.  We conclude that $w(G)\le w(\ell(F)-5)+5w(5)$.
Therefore, assume that $G$ does not have the property (f).

Let us now prove the following claim:
\claim{cl-4chordcol}{Let $H$ be a peeling of $G$ with the outer face $B$, and $\psi$ a precoloring of $B$ that does not extend to an $L$-coloring of $H$.
Let $Q=v_0v_1v_2v_3v_4$ be a $4$-chord of $B$ such that a cycle $C\neq B$ in $B\cup Q$ bounds a face of $G$.
Then, $L(v_1)\subseteq L(v_2)\cup \{\psi(v_0)\}$.}
\begin{proof}
Suppose for a contradiction that there exists a color $c\in L(v_1)\setminus (L(v_2)\cup \{\psi(v_0)\})$.
Let $d$ be a new color that does not appear in the lists of any of the vertices of $V(H)\setminus V(B)$.
Let $N_1\subseteq V(H)\setminus V(B)$ be the set of vertices that are adjacent to $v_1$
and $N_2\subseteq V(H)\setminus V(B)$ be the set of vertices that are adjacent to $v_4$.
Note that $N_1$ and $N_4$ are disjoint, since $v_1$ and $v_4$ do not have a common neighbor other than $v_0$.

If $v_0$ is adjacent to $v_4$, then let $H_1=H-v_0v_4$, otherwise let $H_1=H-(V(C)\setminus V(Q))$.
Let $H_2$ be the graph obtained from $H_1$ by identifying $v_1$ with $v_4$ to a new vertex $v$,
and let $H_3=H_2-vv_2$.  Let $B'$ be the outer face of $H_2$.
Let $L'$ be the list assignment obtained from $L$ by replacing the color $c$ in the lists of vertices of $N_1$
and the color $\psi(v_4)$ in the lists of vertices of $N_2$ by $d$.
Let $H'$ be a $B'$-skeleton of $H_3$ with respect to the list assignment $L'$.
Note that $\ell(B')=\ell(F)+5-\ell(C)\le\ell(F)$ and $|E(H')|<|E(G)|$.

Suppose that $H'$ contains a cycle $K'$ of length at most $4$.  Note that $v\in V(K')$ and $K'$ corresponds
to a path $P$ of length $\ell(K')$ between $v_1$ and $v_4$ in $H$ such that $v_1v_2\not\in E(P)$.
Since the girth of $G$ is at least $5$, the shortest path in between $v_1$ and $v_2$ in $H-v_1v_2$
has length at least $4$, thus $v_2\not\in V(P)$.  It follows that $P\cup v_1v_2v_3v_4$ contains
a cycle $K$ of length at most $7$ and $v_1v_2v_3$ is a subpath of $K$.  By (\ref{cl-face}), such a cycle
bounds a face, implying that $v_2$ has degree two.  This is a contradiction, thus $H'$ has girth at
least $5$.  Similarly, we conclude that $\ell(f_0)\ge 6$.

Let $\psi'$ be the precoloring of $B'$ that matches $\psi$ on $V(B')\setminus\{v\}$ and $\psi'(v)=d$.
Suppose that $\psi'$ extends to an $L'$-coloring of $H'$, and thus also to an $L'$-coloring $\varphi'$ of
$H_3$.  The color $c$ was chosen so that $c\not\in L(v_2)$, and thus $d\not\in L'(v_2)$.  It follows that no vertex of $H_3$ except for $v$ is colored by $d$.
Also, no vertex of $N_1\cup \{v_0\}$ is colored by $c$ and no vertex of $N_2$ is colored by $\psi(v_4)$.
Therefore, the coloring $\varphi$ given by $\varphi(v_1)=c$ and $\varphi(w)=\varphi'(w)$ for $w\in V(H)\setminus (V(B)\cup \{v_1\})$
is an $L$-coloring of $G$ extending $\psi$, which is a contradiction.  We conclude that $\psi'$ does not
extend to an $L'$-coloring of $H'$, and thus $H'\not\in E_1$.  Since $G$ does not have properties (a) and (b),
$B'$ does not have a chord and no vertex of $H'$ has more than two neighbors in $B'$, thus $H'\not\in E_2\cup E_3$,
and $H'$ is not exceptional and $\ell(B')\ge 10$.

As $H'$ has fewer edges than $G$, by the induction hypothesis we get $w(H')\le w(\ell(B')-5)+5w(5)$.
Therefore, every face $f\in \FF(H')$ has length at most $\ell(f)\le \ell(B')-5=\ell(F)-\ell(C)<\ell(F)$.

Consider $H'$ as a subgraph of $H_2$.  Let $f_0$ be the face of $H'$
such that the edge $vv_2$ of $H_2$ is drawn in the open disk bounded by $f_0$, and let $K_0$ be the cycle in $H$ obtained
from $f_0$ by replacing $v$ by the path $C-\{v_2,v_3\}$.  For a cycle $K\subseteq H$, let $H(K)$ be the subgraph of $H$ drawn in the closed disk bounded by $K$.
Note that $w(H)=w(H(K_0))+\sum_{f\in \FF(H')\setminus\{f_0\}} w(H(f))$.  For each face $f\in \FF(H')\setminus \{f_0\}$, the induction
hypothesis implies $w(H(f))\le w(\ell(f))$.  As $v_1v_2\in E(H(K_0))$, we have $H(K_0)\not\in E_1$.
Since $\ell(K_0)=\ell(f_0)+\ell(C)-3\le \ell(F)-3<\ell(F)$, by the induction hypothesis we have
$w(H(K_0))\le w(\ell(f_0)+\ell(C)-6)+w(5)\le w(\ell(f_0))+\ell(C)-6+2w(5)$.
Therefore, \begin{eqnarray*}
w(H)&=&w(H(K_0))+\sum_{f\in \FF(H')\setminus\{f_0\}} w(H(f))\\
&\le&\left(w(\ell(f_0)) + \sum_{f\in \FF(H')\setminus\{f_0\}} w(\ell(f))\right)+\ell(C)-6+2w(5)\\
&=&w(H')+\ell(C)-6+2w(5)\\
&\le& w(\ell(B')-5)+\ell(C)-6+7w(5)\\
&=&w(\ell(F)-\ell(C))+\ell(C)-6+7w(5)\\
&\le& w(\ell(F)-5)+8w(5)-1.
\end{eqnarray*}
It follows that $w(G)\le w(\ell(F)-5)+12w(5)-1\le w(\ell(F)-5)+5w(5)$, which is a contradiction.
\end{proof}

Since $G$ does not have properties (a-f), we conclude that $G$ has properties (g) or (h), i.e.,
there exists a peeling $H$ of $G$ with the outer face $B$, and
a $4$-chord $Q=v_0v_1v_2v_3v_4$ of $B$ such that a cycle $C\neq B$ in $B\cup Q$ bounds a face $C$, and $H$ contains either
\begin{itemize}
\item a jump with base $w_1w_2w_3w_4v_0$ and body $w_1x_1x_2v_1v_0$, or
\item two $5$-faces $w_1w_2v_0v_1x_1$ and $w_3w_4v_4v_3x_2$, with $w_1w_2v_0, w_3w_4v_4\subseteq B$.
\end{itemize}
See Figure~\ref{fig-last} for an illustration.
In the former case, let $R=\{w_2,w_3,w_4,v_0\}$.  In the latter case, let $R=\{w_2,w_4,v_0,v_4\}$.
\begin{figure}
\center{\epsfbox{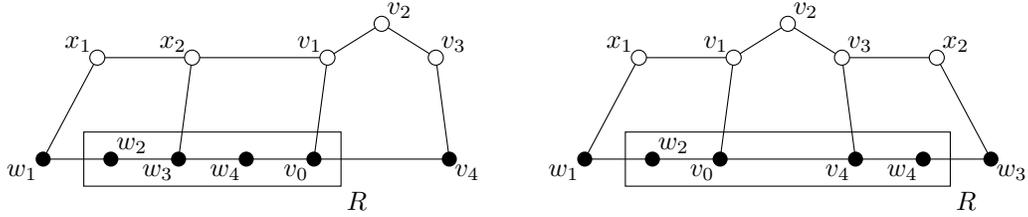}}
\caption{The configurations in properties (g) and (h).}
\label{fig-last}
\end{figure}

As $H\neq B$ is $B$-critical, there exists a precoloring $\psi$ of $B$ that does not extend to an $L$-coloring of $H$.
Let $X_1=L(v_1)\setminus \{\psi(v_0)\}$ and $X_3=L(v_3)\setminus\{\psi(v_4)\}$.  By (\ref{cl-4chordcol}),
$X_1\cup X_3\subseteq L(v_2)$, and since $|X_1|,|X_3|\ge 2$ and $|L(v_2)|=3$, there exists a color $c\in X_1\cap X_3$.
Let $H_1=G-(V(C)\setminus V(Q))-R$ and let $H_2$ be the graph obtained from $H_1$ by identifying $v_1$ with $v_3$ to a new vertex $v$.
Let $B'$ be the outer face of $H_2$.  Let $H'$ be a $B'$-skeleton of $H_2$, with respect to the restriction of $L$ to $V(H_2)\setminus V(B')$.
Note that $\ell(B')=\ell(F)-\ell(C)+4<\ell(F)$.

Suppose that $H'$ contains a cycle $K$ of length at most $4$.  Then $v\in V(K)$ and $H_1$ contains
a path $P$ of length $\ell(K)$ between $v_1$ and $v_3$.  Note that $v_2\not\in V(P)$, since the girth of $G$ is at least $5$.
Therefore, $P\cup v_1v_2v_3$ is a cycle of length at most $6$, and by (\ref{cl-face}), it bounds a face.  It follows that
$v_2$ has degree two, which is a contradiction.  It follows that $H'$ has girth at least $5$.

Let $\psi'$ be the precoloring of $B'$ that matches $\psi$ on $V(F)\cap V(B')$, with $\psi'(v)=c$
and the colors of $x_1\in L(x_1)$ and $x_2\in L(x_2)$ chosen so that $\psi'$ is a proper coloring of $B'$.
Suppose that $\psi'$ extends to an $L$-coloring of $H'$, and thus also to an $L$-coloring $\varphi'$ of $H_2$.
Setting $\varphi(v_1)=\varphi(v_3)=c$ and $\varphi(z)=\varphi'(z)$ for $z\in V(H)\setminus(V(B)\cup \{v_1,v_3\})$,
we obtain an $L$-coloring of $H$ extending $\psi$, which is a contradiction.  Therefore, $\psi'$ does not
extend to $G'$, and $G'\not\in E_1$.  Furthermore, since $G$ does not have properties (a), (b) and (d),
$B'$ has no chords and no vertex of $H'$ has more than two neighbors in $B'$, hence $H'$ is not exceptional
and $\ell(B')\ge 10$.

By the induction hypothesis, $w(H')\le w(\ell(B')-5)+5w(5)=w(\ell(F)-\ell(C)-1)+5w(5)$.
It follows that each face of $H'$ has length at most $\ell(B')-5=\ell(F)-\ell(C)-1<\ell(F)$.

Consider $H'$ as the subgraph of $H_2$.  If $vv_2\not\in E(H')$, then let $f_0$ be the face of $H'$ such that the closed
disk bounded by $f_0$ contains the edge $vv_2$.  Let $K_0\subseteq H$ be the cycle obtained from $f_0$ by replacing $v$ by the path $C-v_2$.
Since $v_1v_2,v_3v_2\in E(H(K_0))$, it follows that $H(K_0)\not\in E_0\cup E_1$ and $\ell(K_0)\ge 9$.
Also, $\ell(K_0)=\ell(f_0)+\ell(C)-2\le \ell(F)-3$.  By the induction hypothesis,
$w(H(K_0))\le w(\ell(K_0)-4)+2w(5)=w(\ell(f_0)+\ell(C)-6)+2w(5)\le w(\ell(f_0))+\ell(C)-6+3w(5)$.
Also, for each $f\in \FF(H')\setminus \{f_0\}$, $w(H(f))\le w(\ell(f))$.
If $vv_2\in E(H')$, then we let $K_0=C$ and $w(H(K_0)))=w(C)$.  Note that in addition to the faces
contained in the graphs $H(f)$ for $f\in \FF(H')\setminus \{f_0\}$ and in $H(K_0)$, $H$ has two more $5$-faces.
Since $\ell(C)-6+3w(5) < w(C)$, we conclude that $w(H)\le w(H')+w(C)+2w(5)\le w(\ell(F)-\ell(C)-1)+w(C)+7w(5)$.

If $\ell(F)-\ell(C)=6$, then, since $\ell(F)\ge 12$, we have $\ell(C)\ge 6$ and
$w(\ell(F)-\ell(C)-1)+w(C)=(\ell(F)-\ell(C)-6+w(5))+(\ell(C)-5)=\ell(F)-11+w(5)=w(\ell(F)-5)+w(5)-1$.
If $\ell(F)-\ell(C)>6$, then $w(\ell(F)-\ell(C)-1)+w(C)\le (\ell(F)-\ell(C)-6))+(\ell(C)-5+w(5))=w(\ell(F)-5)+w(5)-1$.
Therefore, $w(H)\le w(\ell(F)-5)+8w(5)-1$ and $w(G)\le w(\ell(F)-5)+12w(5)-1\le w(\ell(F)-5)+5w(5)$.
\end{proof}

Theorem~\ref{thm-fincrit} implies that the number of vertices of a $F$-critical plane graph of girth at least $5$
is linear in $\ell(F)$:

\begin{proof}[Proof of Theorem~\ref{thm-numvert}]
If $G$ is exceptional, then $|E(G)|\le \ell(F)+3 < 18\ell(F)-160$,
and $|V(G)|\le \ell(F)+1 < \frac{37\ell(F)-320}{3}$, since $\ell(F)\ge 10$.
Therefore, assume that $G$ is not exceptional.

For each $x\ge 5$, we have $w(x)\ge w(5)x/5$.
By Theorem~\ref{thm-fincrit},
\begin{eqnarray*}
2w(5) |E(G)|/5&=&w(5) \ell(F)/5 + \sum_{f\in \FF(G)} w(5)\ell(f)/5\\
&\le& w(5) \ell(F)/5 + \sum_{f\in\FF(G)}w(\ell(f))\\
&=& w(5) \ell(F)/5 + w(G)\\
&\le& w(5) \ell(F)/5 + w(\ell(F)-5)+5w(5)\\
&\le& (1+w(5)/5)\ell(F)-10+6w(5).  
\end{eqnarray*}
Therefore, $|E(G)|\le (1+5/w(5))\ell(F)/2-25/w(5)+15=18\ell(F)-160$.

As the minimum degree of $G$ is at least $2$ and all vertices except for those in
$f$ have degree at least three, we get
$3|V(G)|-\ell(F)\le 2|E(G)|\le 36\ell(F)-320$, and hence $|V(G)|\le \frac{37\ell(F)-320}{3}$.
\end{proof}

\section{Concluding remarks}
The bound on $|V(G)|$ in Theorem~\ref{thm-numvert} can be improved by $\ell(F)/6$ by
a slightly more involved argument, first eliminating $\le 2$-chords and edges joining vertices of degree two.
However, the bound seems to be far from the correct one for large values of $\ell(F)$.

As the number of vertices of an $F$-critical graph is linear in $\ell(F)$, the number of such graphs
is at most exponential in $\ell(F)$ (Denise et al.~\cite{numplan}).  On the other hand, every tree with $k$
leaves and all internal vertices of degree three gives rise to an $F$-critical graph with $\ell(F)=3k$, thus
the number of $F$-critical graphs is exponential in $\ell(F)$.

The proof of Theorem~\ref{thm-fincrit} can be converted to an algorithm to generate the critical graphs
in the straightforward way---each critical graph $G$ contains a configuration described by Lemma~\ref{lemma-struct},
and this configuration can be used to derive $G$ from smaller critical graphs.   This algorithm could be practical
for small values of $\ell(F)$, say $\ell(F)<20$.

A slightly unsatisfactory part of the proof of Theorem~\ref{thm-fincrit} concerns dealing with the cases (g) and (h)
of Lemma~\ref{lemma-struct}, where the reduced graph $H'$ is not a subgraph of $G$ drawn inside a cycle of $G$.
It would be more appealing to have a proof that avoids such non-trivial reductions, giving a better understanding of the structure
of the critical graphs, as well as a faster algorithm to generate them.

\section*{Acknowledgements}

We would like to thank Bernard Lidick\'y for his help with improving the presentation
of this paper.

\end{document}